\documentclass[a4paper,10pt,reqno]{amsart}

\usepackage[T1]{fontenc}
\usepackage{caption}
\usepackage[a4paper,top=2cm,bottom=2cm,left=2.2cm,right=2.2cm]{geometry}

\usepackage{titlesec}
\usepackage{lipsum}
\usepackage{enumerate}
\usepackage{enumitem}
\usepackage{mathrsfs}
\usepackage{amsmath,amssymb,amsthm,bm}
\usepackage{mathtools}
\usepackage{graphicx, float, subfigure}
\usepackage{url}
\usepackage{tikz}
\usepackage{bbold}
\usepackage{color}
\usepackage{xcolor}

\usepackage[all,cmtip]{xy}
\usepackage{multirow}
\usepackage{makecell}
\usepackage[colorinlistoftodos]{todonotes}
\usepackage[colorlinks=true, allcolors=red]{hyperref}
\usepackage{harpoon}
\usepackage{MnSymbol}
\usepackage{esvect}
\usepackage{diagbox}
\usepackage[title]{appendix}

\DeclareMathOperator{\sgn}{sgn}
\DeclareMathOperator{\wt}{wt}

\theoremstyle{plain}
\newtheorem{theorem}{\scshape Theorem}[section]

\newtheorem{lemma}[theorem]{\scshape Lemma}
\newtheorem{corollary}[theorem]{\scshape Corollary}

\newtheorem*{assumption*}{\scshape Assumption}

\newtheorem*{claim*}{Claim}

\theoremstyle{definition}

\newtheorem{definition}[theorem]{\scshape Definition}
\newtheorem{remark}[theorem]{\scshape Remark}

\newcommand{\M}{\operatorname{M}}

\renewcommand{\det}{\mathsf{det}}
\renewcommand{\wt}{\mathsf{wt}}
\renewcommand{\r}{\mathbf{r}}
\renewcommand{\b}{\mathsf{block}}
\newcommand{\x}{\mathbf{x}}

\numberwithin{equation}{section}

\titleformat{\section}{\centering\bfseries}{\thesection}{1em}{\MakeUppercase}
\titleformat{\subsection}{\bfseries}{\thesubsection}{1em}{}

\begin{document}
\title{Block diagonally symmetric lozenge tilings}
\thanks{S.H.B. was supported in part by an AMS-Simons Travel Grant.}
%
%
\author{Seok Hyun Byun and Yi-Lin Lee}
\address{(S. H. Byun) Department of Mathematics, Amherst College, Amherst, Massachusetts 01002, U.S.A.}
\email{sbyun@amherst.edu}

\address{(Y.-L. Lee) Department of Mathematics, National Cheng Kung University, Tainan 70101, Taiwan}
\email{yillee@gs.ncku.edu.tw}

%

\begin{abstract}
We introduce a new symmetry class of both plane partitions and lozenge tilings of a hexagon, called the \textit{$\mathbf{r}$-block diagonal symmetry class}, where $\mathbf{r}$ is an $n$-tuple of non-negative integers. We prove that the tiling generating function of this symmetry class under a certain weight assignment is given by a simple product formula. When $\mathbf{r}$ and weights are suitably chosen, our formula coincides with the number of domino tilings of the Aztec diamond and the number of alternating sign matrices up to a simple constant. We also deduce the volume generating function of $\mathbf{r}$-block diagonally symmetric plane partitions. Additionally, we consider \textit{$(\mathbf{r},\mathbf{r^{\prime}})$-block diagonally symmetric} lozenge tilings by embedding the hexagon into a cylinder and present an identity for the signed enumeration of this symmetry class in specific cases. Two methods are provided to study this symmetry class: (1) the method of non-intersecting lattice paths with a modification and (2) interpreting weighted lozenge tilings algebraically as (skew) Schur polynomials and applying the dual Pieri rule.
\end{abstract}
\maketitle

\section{Introduction}\label{sec:intro}

Consider a triangular lattice oriented so that one family of lattice lines is vertical. A \textit{lozenge} is a union of two adjacent unit triangles on the lattice. A \textit{lozenge tiling} of a region on the lattice is a collection of lozenges that covers the region without gaps or overlaps. In enumerative combinatorics, it is standard to assign weights to individual lozenges such that the weight of a tiling---defined as the product of weights of all lozenges that constitute the tiling---yields a meaningful combinatorial interpretation.

Throughout this paper, we write $\M_{\wt}(H)$ for the \textit{tiling generating function} of $H$ under weight assignments on individual lozenges in $H$ determined by a weight function $\wt$, which is the sum of the weights of all tilings of $H$. We will clearly specify different weight functions $\wt$ that are considered in the paper. If all the weights are $1$, then we simply write $\M(H)$ for the number of tilings of $H$.

An \textit{$(a \times b \times n)$-plane partition} is an $a \times b$ array $\pi=(\pi_{i,j})$ of non-negative integers such that $\pi_{i,j} \leq n$ and these integers are weakly decreasing along rows and columns, that is, $\pi_{i,j} \geq \pi_{i+1,j}$ and $\pi_{i,j} \geq \pi_{i,j+1}$. Let $PP^n( a \times b)$ denote the set of $(a \times b \times n)$- plane partitions. Let $H(n,a,b)$ be a hexagon on the triangular lattice with sides of length $n$, $a$, $b$, $n$, $a$, and $b$ in clockwise order from the left side. The classical result of MacMahon \cite{MacM} on the enumeration of $PP^n( a \times b)$ can be rephrased as the enumeration of lozenge tilings of $H(n,a,b)$, via the bijection due to David and Tomei \cite{DT} (see Figure \ref{fig:PPcorrespondence}). This number is given by the following simple product formula:
\begin{equation}\label{eq.MacMbox}
    |PP^{n}(a \times b)| = \M(H(n,a,b)) = \prod_{i=1}^{a}\prod_{j=1}^{b}\frac{n+i+j-1}{i+j-1}.
\end{equation}

Since a plane partition can be viewed as a pile of unit cubes (compare the left and the middle pictures in Figure \ref{fig:PPcorrespondence}), it is natural to consider a weighted enumeration of $PP^n( a \times b)$ by the \textit{volume} of these piles. MacMahon \cite{MacM} established the following volume generating function of $PP^n( a \times b)$:
\begin{equation}\label{eq.MacMboxq}
    \sum_{\pi \in PP^n(a\times b)}q^{|\pi|} = \prod_{i=1}^{a}\prod_{j=1}^{b} \frac{1-q^{n+i+j-1}}{1-q^{i+j-1}},
\end{equation}
where $|\pi| = \sum_{i,j}\pi_{i,j}$ is the sum of all the entries of $\pi$, which is equivalent to the volume of the pile of unit cubes. On the other hand, the above volume generating function, up to a simple multiplicative factor, can be reinterpreted as the sum of weighted lozenge tilings of $H(n,a,b)$ by assigning the weight $q^k$ to each horizontal lozenge where the distance between the southeastern side of such lozenge and the southeastern side of $H(n,a,b)$ is $\frac{\sqrt{3}}{2}k$ (see the right picture in Figure \ref{fig:PPcorrespondence}, where shaded lozenges labeled by $k$ are weighted by $q^k$). The other lozenges are weighted by $1$.
\begin{figure}[hbt!]
    \centering \includegraphics[width=0.6\textwidth]{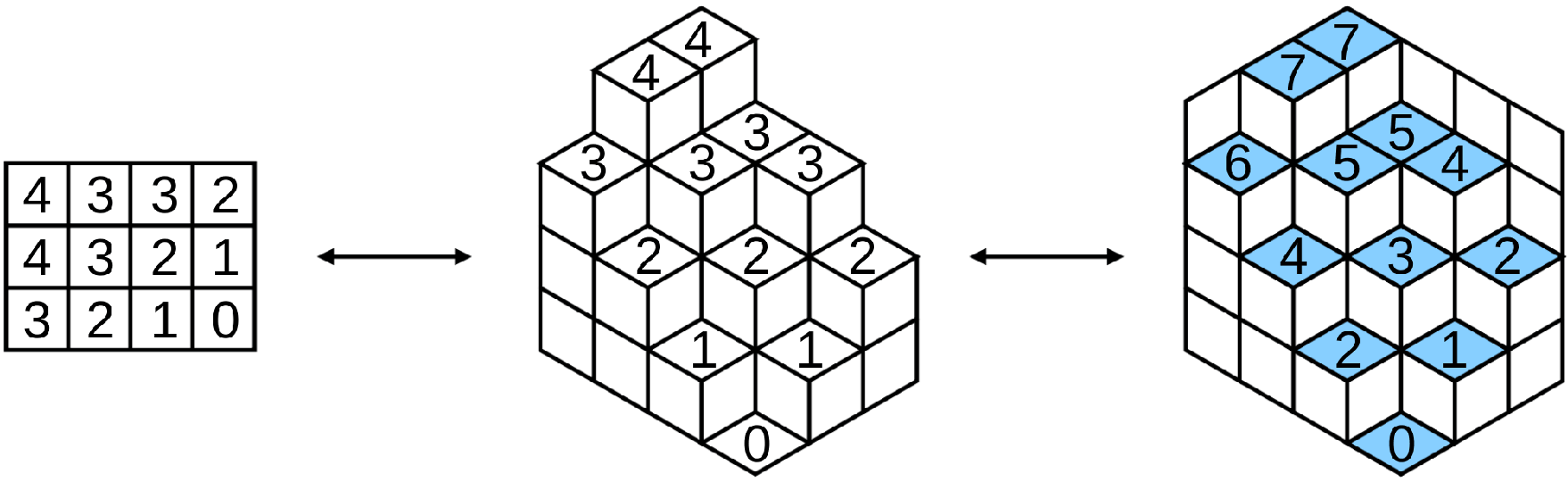}
\caption{The correspondence between a plane partition, its 3D realization as a pile of unit cubes, and a lozenge tiling of $H(4,3,4)$ (from left to right). Note that if we add $3$, $2$, $1$, and $0$ to the numbers in the first, second, third, and last columns in the plane partition, respectively, we obtain the numbers marked on the shaded lozenge in the right picture. Each shaded lozenge marked by $k$ is weighted by $q^k$.}
\label{fig:PPcorrespondence}
\end{figure}

In the 1970s and 80s, the study of symmetry classes of plane partitions drew the attention of researchers in enumerative combinatorics. Enumerating symmetry classes is difficult in general since the structure of each class varies with different symmetry axes. These usually require different methods to enumerate them. Stanley \cite{Stan86} identified ten symmetry classes of plane partitions, and remarkably, each of them is enumerated by elegant product formulas. We refer the reader to the survey paper by Krattenthaler \cite[Section 6]{Kra16} for a modern update to Stanley’s paper.

For instance, one of the symmetry classes, called the \textit{(diagonally) symmetric} plane partitions, contains plane partitions $\pi \in PP^n(m \times m)$ that are invariant under taking the transpose, i.e. $\mathsf{tr}(\pi) = \pi$. This can be viewed as lozenge tilings of $H(n,m,m)$ that are symmetric about the vertical diagonal, and a generating function of this symmetry class is given by the following product formula; see \cite[Section 2.2]{HL21} and references therein:
\begin{equation}\label{eq.MacMsscq1}
    \sum_{\substack{\pi \in PP^n(m \times m) \\ \mathsf{tr}(\pi) = \pi}}q^{|\pi|'} = \prod_{1 \leq i \leq j \leq m} \frac{1-q^{n+i+j-1}}{1-q^{i+j-1}},
\end{equation}
where $|\pi|' = \sum_{i \leq j}\pi_{i,j}$ is a kind of ``half-size'' of $\pi$. Taking the limit $q\rightarrow1$, then \eqref{eq.MacMsscq1} enumerates this symmetry class:
\begin{equation}\label{eq.MacMssc}
    |\{ \pi \in PP^n(m \times m )~|~ \mathsf{tr}(\pi) = \pi \}| = \prod_{1 \leq i \leq j \leq m} \frac{n+i+j-1}{i+j-1}.
\end{equation}

In this paper, we introduce a new symmetry class, called the \textit{$\r$-block diagonal symmetry class}, of both plane partitions (see Definition \ref{def.blocksympp}) and lozenge tilings of a hexagonal region (see Definition \ref{def.rblock}). This symmetry class is motivated by the second author's previous work \cite{Lee1} on off-diagonally symmetric domino tilings of Aztec diamonds. The first main result (as shown in Theorem \ref{thm:1}) presents a product formula for the weighted enumeration of $\r$-block diagonally symmetric lozenge tilings of $H(n,m,m)$, where our weight generalizes the ``volume'' of MacMahon's plane partitions. When $\mathbf{r}$ and weights are suitably chosen, our formula coincides with the number of domino tilings of the Aztec diamond (when $\mathbf{r}=(1,\ldots,1)$ and specializing all weights to $1$) and the number of alternating sign matrices (when $\mathbf{r}=(2,\ldots,2)$ and specializing all weights to $1$) up to a simple constant; see the paragraph after Corollary \ref{cor:1}. We also deduce the volume generating function of $\r$-block diagonally symmetric plane partitions (Theorem \ref{thm:1.5}). 

Moreover, we generalize the idea of $\r$-block diagonally symmetric lozenge tilings to \textit{$(\r,\r^{\prime})$-block diagonally symmetric} lozenge tilings, which can be viewed as lozenge tilings of the region obtained by gluing two opposite vertical sides of $H(n,m,m)$ and embedding it into a cylinder. We impose our ``block diagonal symmetry'' condition on the vertical symmetry axis of $H(n,m,m)$ and on the two opposite vertical sides of $H(n,m,m)$ that were glued. Our second result (as shown in Theorem \ref{thm:2}) provides an identity involving weighted $(\r,\r^{\prime})$-block diagonally symmetric lozenge tilings for particular cases.

We provide two methods to study $\r$-block diagonally symmetric lozenge tilings.

The first method is based on the method of non-intersecting paths (see \cite[Section 3.1]{Propp15}), which is one of the widely used techniques for enumerating domino or lozenge tilings. This requires the following three steps:
\begin{itemize}
    \item[(1)] Establish a one-to-one correspondence between tilings and families of non-intersecting lattice paths.

    \item[(2)] Apply the Lindstr\"om--Gessel--Viennot (LGV) theorem (\cite{Lind} and \cite{GV}) or Stembridge's Pfaffian generalization (\cite{Stem90}) to construct matrices whose determinants or Pfaffians give the desired number. 
    
    \item[(3)] Evaluate the determinants or Pfaffians. In this step, people try to find closed-form expressions.
\end{itemize}
Step (2) requires specific conditions on the starting and ending points of families of non-intersecting lattice paths (see Section \ref{sec:latticepath} for details). Step (3) is more difficult in general. Deriving a closed-form expression is not always possible; successful evaluations typically require either sophisticated algebraic manipulations of the matrix entries or a structural insight into the matrices (for example, see \cite{Krattenthaler} and \cite{Krattenthaler2}).

Unfortunately, neither of the theorems can be directly applied to enumerate this symmetry class due to the $\r$-block diagonal symmetry condition. One of our main contributions is to resolve this issue; we modify the corresponding lattice paths so that the LGV theorem works.

The second method is inspired by the relation between weighted lozenge tilings of a trapezoidal region with certain boundary defects and Schur polynomials (see \cite[Theorem 2.3]{AF}). The $\r$-block diagonally symmetric weighted lozenge tilings are given by a sum of Schur polynomials $s_{\lambda}$ for some particular partitions $\lambda$. We then obtain our product formula by the dual Pieri rule (see, for example, \cite[Page 340]{Stanley}) and the specialization of Schur polynomials. A similar proof based on the dual Pieri rule is given for $(\r,\r^{\prime})$-block diagonally symmetric lozenge tilings (Theorem \ref{thm:2}).

The paper is organized as follows. In Section \ref{sec:main}, we introduce the $\r$-block diagonal symmetry class, $(\r,\r^{\prime})$-block diagonal symmetry class, and then state our main results. We provide two methods to study these symmetry classes. In Section \ref{sec:latticepath}, we present the first method, which is based on the method of non-intersecting lattice paths with a modification, and the proof of our first main result (Theorem \ref{thm:1}). In Section \ref{sec:algproof}, we provide the second method by expressing weighted lozenge tilings as (skew) Schur polynomials and using the dual Pieri rule. An alternative proof of the first main result and the proof of our second main result (Theorem \ref{thm:2}) are given.

\section{Statement of main results}\label{sec:main}

We will make use of the following notations. For any positive integer $k$, define $[k]_q:= \frac{1-q^k}{1-q}$ and $[k]_q! :=[k]_q \cdots [1]_q$ with $[0]_q:=1$ and $[0]_q!:= 1$. The $q$-binomial coefficient $\begin{bmatrix} n\\k \end{bmatrix}_q$ is defined to be $\frac{[n]_q!}{[k]_q![n-k]_q!}$ if $n\geq k\geq0$ and $0$ otherwise. Lastly, we set $[k]:=\{1,\ldots,k\}$.

In Section \ref{sec:off-diag}, we first review the off-diagonal symmetry class of domino tilings of Aztec diamonds and then extend this notion to lozenge tilings of hexagons. In Section \ref{sec:rblocklozenge}, we introduce $\r$-block diagonally symmetric lozenge tilings and state the first main result (Theorem \ref{thm:1}). In Section \ref{sec:rblockpp}, we introduce $\r$-block diagonally symmetric plane partitions and explain how to deduce the volume generating function of $\r$-block diagonally symmetric plane partitions from Theorem \ref{thm:1}. In Section \ref{sec:rrblock}, we generalize the $\r$-block diagonal symmetry class to the $(\r,\r^{\prime})$-block diagonal symmetry class and present the second main result (Theorem \ref{thm:2}).

\subsection{The off-diagonal symmetry class}\label{sec:off-diag}

The \textit{Aztec diamond of order $n$}, denoted by $AD(n)$, is the union of all unit squares in the region $|x|+|y|\leq n+1$; see Figure \ref{fig:offAD1}, ignore the six marked $2 \times 2$ squares at this point. A \textit{domino} is a $1 \times 2$ or $2 \times 1$ rectangle. A \textit{domino tiling} of $AD(n)$ is a collection of dominoes that covers $AD(n)$ without gaps or overlaps. Elkies, Kuperberg, Larsen, and Propp \cite{EKLP1, EKLP2} proved that the number of domino tilings of $AD(n)$ is $2^{n(n+1)/2}$.

\begin{figure}[hbt!]
    \centering
    \subfigure[]{\label{fig:offAD1}\includegraphics[width=0.25\textwidth]{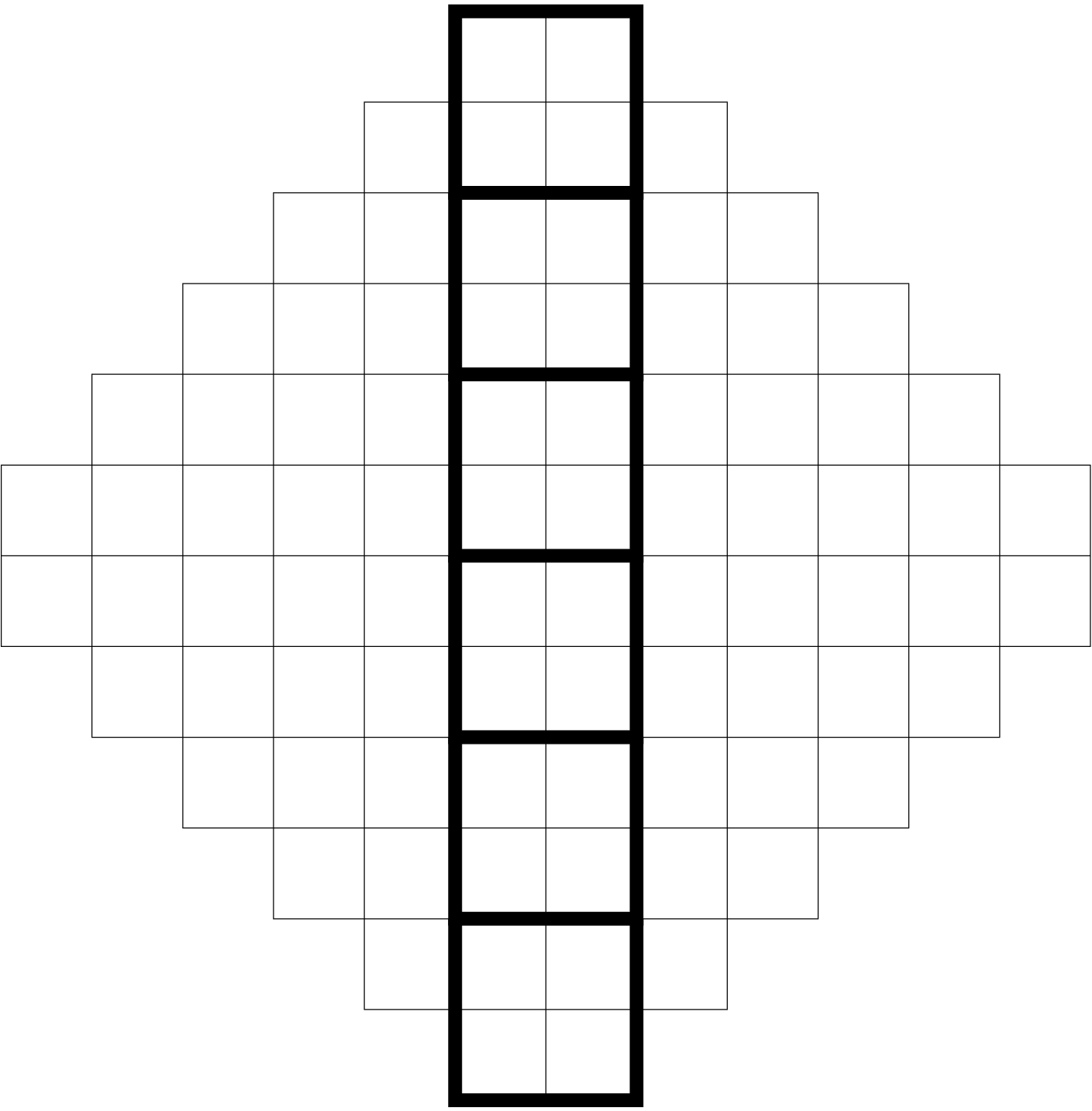}}
    \hspace{15mm}
    \subfigure[]{\label{fig:offAD2}\includegraphics[width=0.25\textwidth]{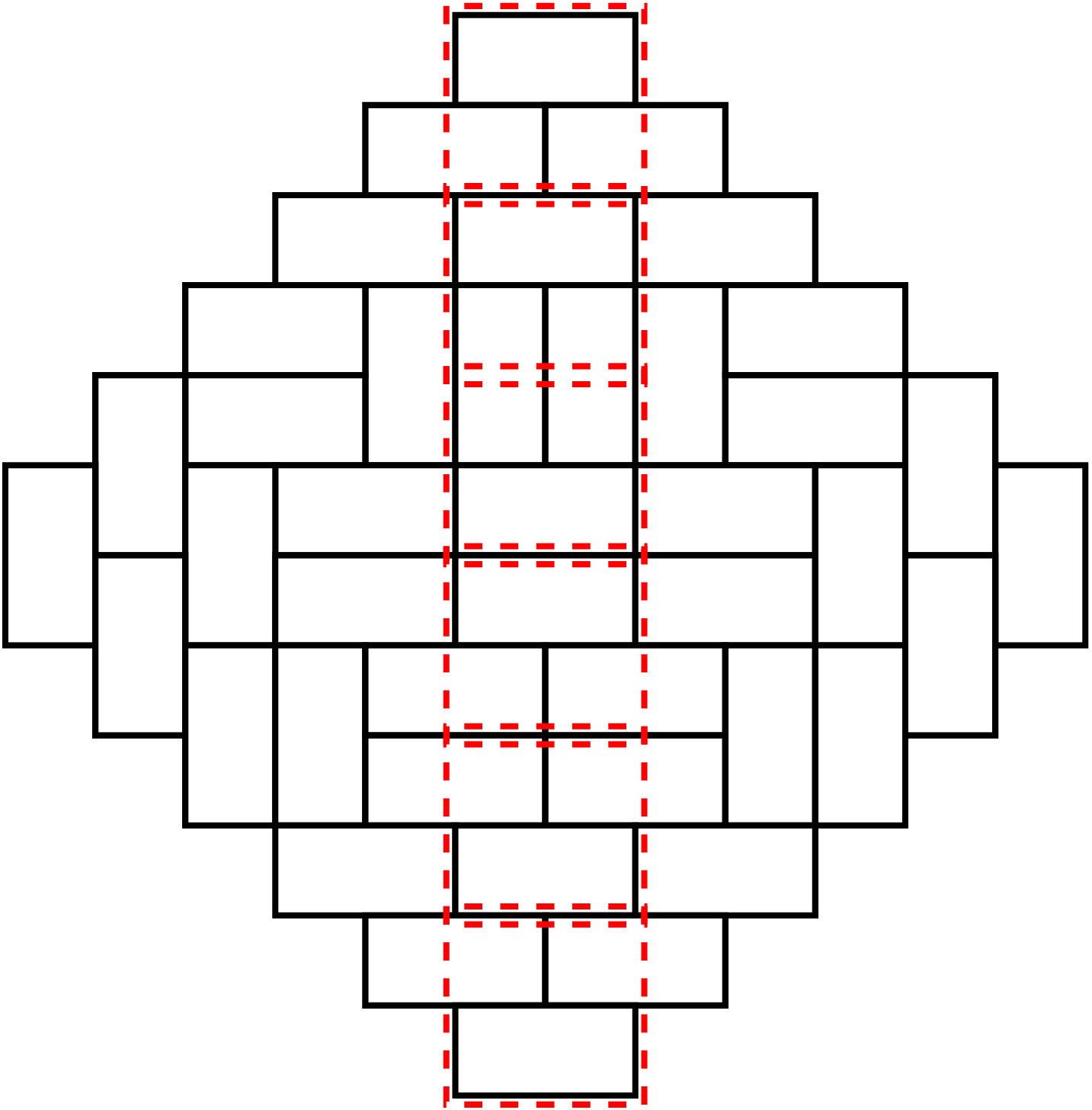}}
    \caption{(a) $AD(6)$ with the six cells marked. (b) An off-diagonally symmetric domino tiling of $AD(6)$.}
\end{figure}

Let $\ell$ be the vertical symmetry axis of $AD(n)$. Notice that there are $n$ non-overlapping $2 \times 2$ squares on $\ell$ that are bisected by $\ell$ (see the six marked squares in Figure \ref{fig:offAD1}). Each of these $n$ non-overlapping $2 \times 2$ squares is called a \textit{cell}. A domino tiling of $AD(n)$ is \textit{off-diagonally symmetric} (\cite[Definition 1]{Lee1}) if
\begin{itemize}
    \item the tiling is symmetric about $\ell$, and
    \item each such cell contains exactly one full horizontal domino.
\end{itemize}
Figure \ref{fig:offAD2} shows an off-diagonally symmetric domino tiling of $AD(6)$.

The second condition from the above definition is a natural analog of \textit{off-diagonally symmetric alternating sign matrix} introduced by Kuperberg \cite{Kuperberg}. In fact, the number of off-diagonally symmetric domino tilings of $AD(n)$ is a weighted enumeration, known as the \textit{$2$-enumeration}, of off-diagonally symmetric alternating sign matrices of size $n$. See \cite[Section 1.2]{Lee1} and \cite[Section 1.1]{Lee2} for more details about how the off-diagonal symmetry classes of these two objects are related. 

The off-diagonal symmetry classes of alternating sign matrices and domino tilings of $AD(n)$ have many nice properties. For example, Kuperberg \cite{Kuperberg} showed that the number of off-diagonally symmetric alternating sign matrices is given by a simple product formula (see also equation (126) in \cite{BFK23}). In the case of off-diagonally symmetric domino tilings of $AD(n)$, the cardinality is given by a Pfaffian of a certain matrix whose entries satisfy a simple recurrence relation \cite[Theorem 2 and Theorem 4]{Lee1}, and also has a symmetry property \cite[Theorem 1.2]{Lee2}. 

Now, we extend the notion of the off-diagonal symmetry to lozenge tilings of hexagons as follows. Let $\ell$ be the vertical symmetry axis of $H(n,n,n)$, the regular hexagon of side length $n$. Note that $\ell$ can be covered by $n$ non-overlapping unit hexagons that are bisected by $\ell$ as shown in Figure \ref{fig:offHex1}. Each such unit hexagon is called a \textit{$1$-cell}.

\begin{figure}[hbt!]
    \centering
    \subfigure[]{\label{fig:offHex1}\includegraphics[width=0.25\textwidth]{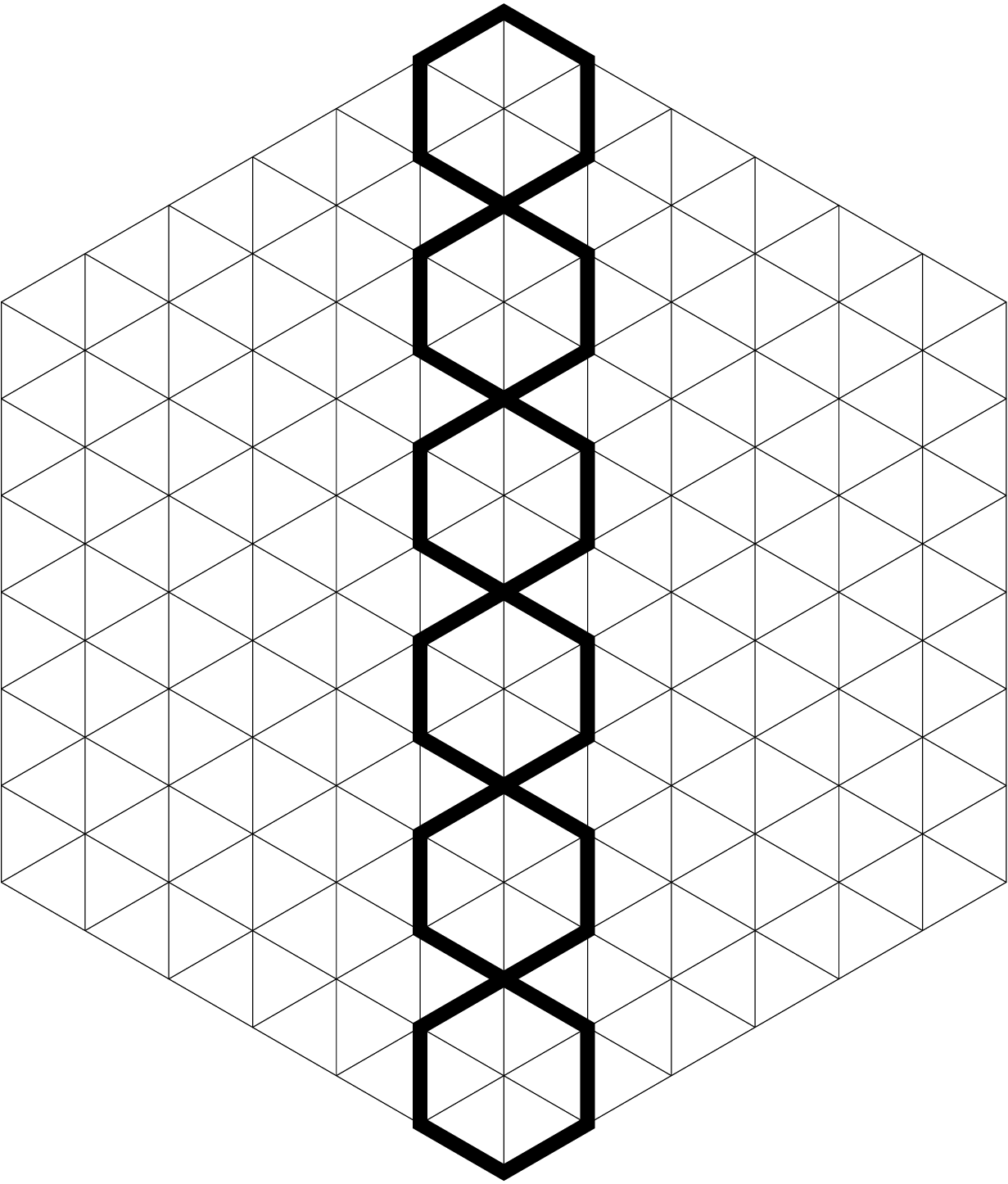}}
    \hspace{15mm}
    \subfigure[]{\label{fig:offHex2}\includegraphics[width=0.25\textwidth]{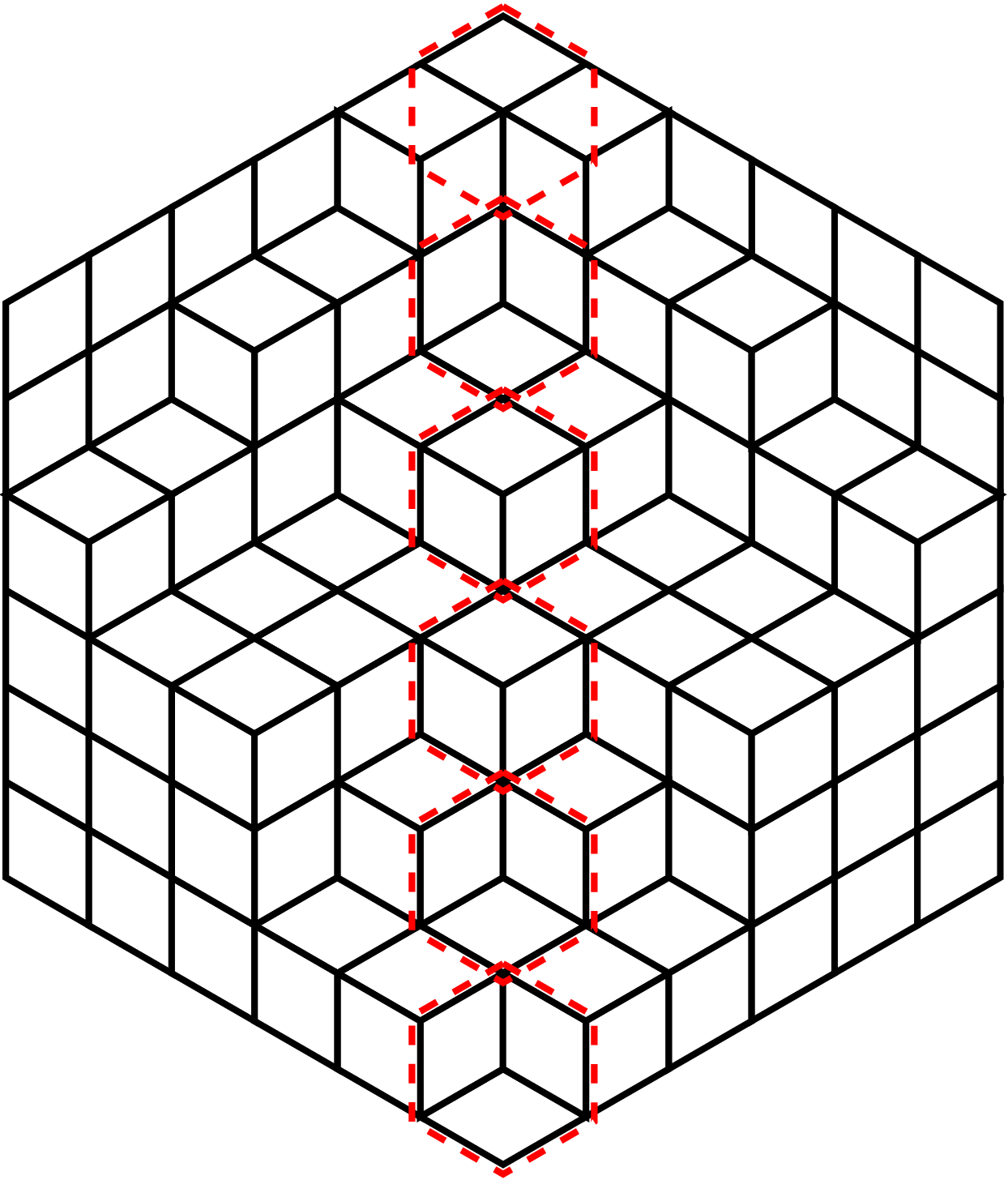}}
    \caption{(a) $H(6,6,6)$ with six $1$-cells marked. (b) An off-diagonally symmetric lozenge tiling of $H(6,6,6)$.}
\end{figure}
\begin{definition}\label{def.offhex}
    A lozenge tiling of $H(n,n,n)$ is \textit{off-diagonally symmetric} if
    \begin{itemize}
        \item the tiling is symmetric about $\ell$, and
        \item each $1$-cell contains exactly one full horizontal lozenge.
    \end{itemize}
\end{definition}
\noindent Figure \ref{fig:offHex2} depicts an off-diagonally symmetric lozenge tiling of $H(6,6,6)$.

We first notice that the number of off-diagonally symmetric lozenge tilings of $H(n,n,n)$ is $2^{n(n+1)/2}$, which coincides with the number of domino tilings of $AD(n)$. This surprising correspondence motivates us to generalize this symmetry class to the $\r$-block diagonal symmetry class (introduced in Section \ref{sec:rblocklozenge}). Indeed, we prove that this result can be obtained as a special case of Corollary \ref{cor:1}. 

\subsection{$\r$-block diagonally symmetric lozenge tilings}\label{sec:rblocklozenge}

For any positive integers $m$ and $n$, the hexagon $H(n,m,m)$ is symmetric across the vertical diagonal $\ell$. Note that the vertical diagonal has length $m+n$. When $m=1$, we call the region $H(k,1,1)$ a \textit{$k$-cell}. Note that when $k=0$, the $0$-cell reduces to a horizontal lozenge, as shown in Figure \ref{fig:blockHex1}.

Consider an $n$-tuple of non-negative integers $\mathbf{r}=(r_1,\dots,r_n)$ such that their sum is equal to $m$, that is, $|\mathbf{r}|\coloneqq r_1+\cdots+r_n=m$. Since the sum $(r_1+1)+\cdots+(r_n+1)=m+n$ is equal to the length of the vertical diagonal of $H(n,m,m)$, we can uniquely place an $r_1$-cell, an $r_2$-cell,..., and an $r_n$-cell on $\ell$ from bottom to top inside $H(n,m,m)$, such that these $n$ cells do not overlap with each other and they are bisected by $\ell$ (see Figure \ref{fig:blockHex1} for an example). We use $S_k = \sum_{i=1}^{k}r_i$ to denote the partial sum of an $n$-tuple $(r_1,\dots,r_n)$ for $k=1,\dots,n$, with the convention $S_{0}=0$. In particular, $S_{n}=m$.

\begin{figure}[hbt!]
    \centering
    \subfigure[]{\label{fig:blockHex1}\includegraphics[width=0.32\textwidth]{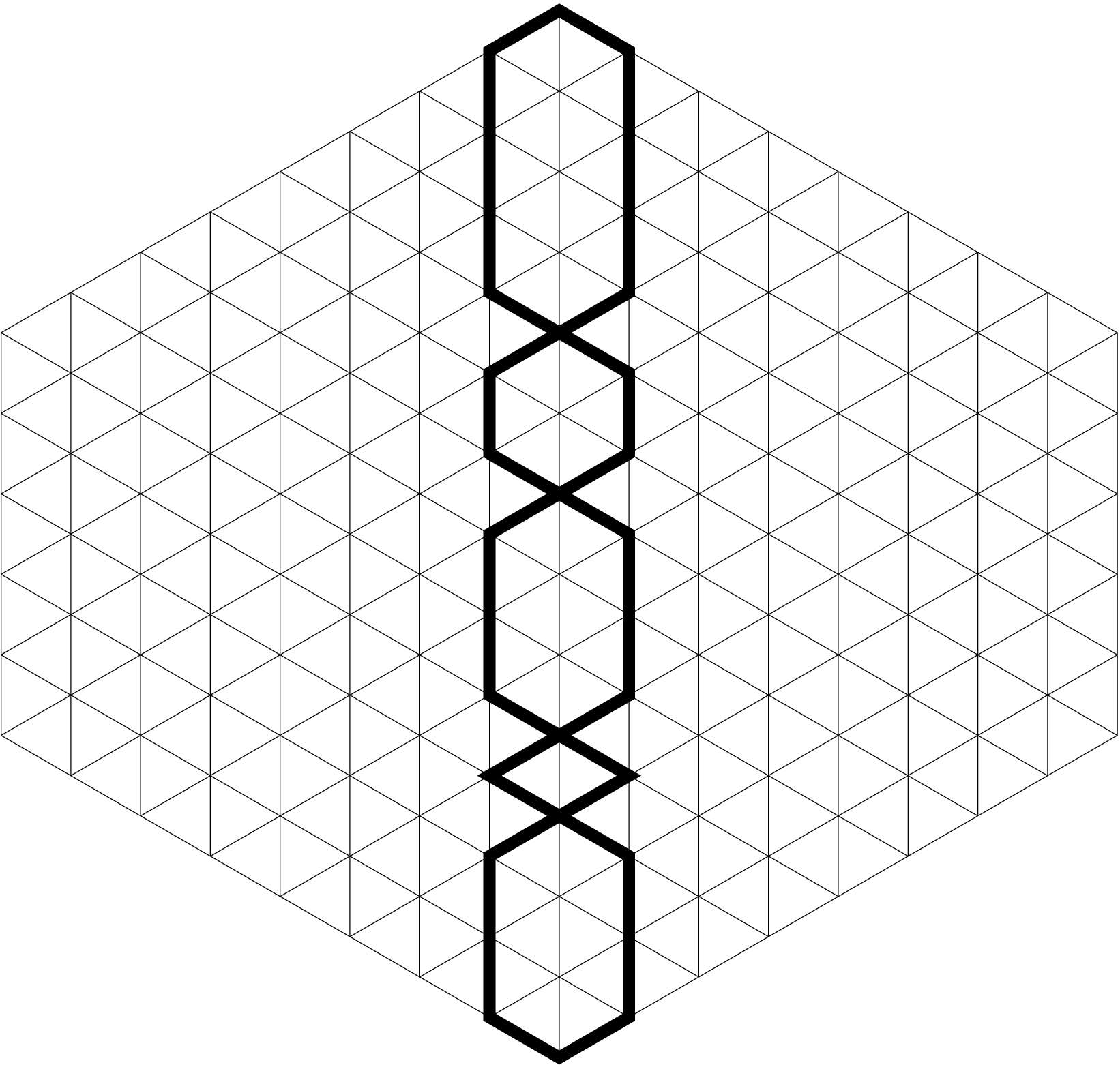}}
    \hspace{15mm}
    \subfigure[]{\label{fig:blockHex2}\includegraphics[width=0.32\textwidth]{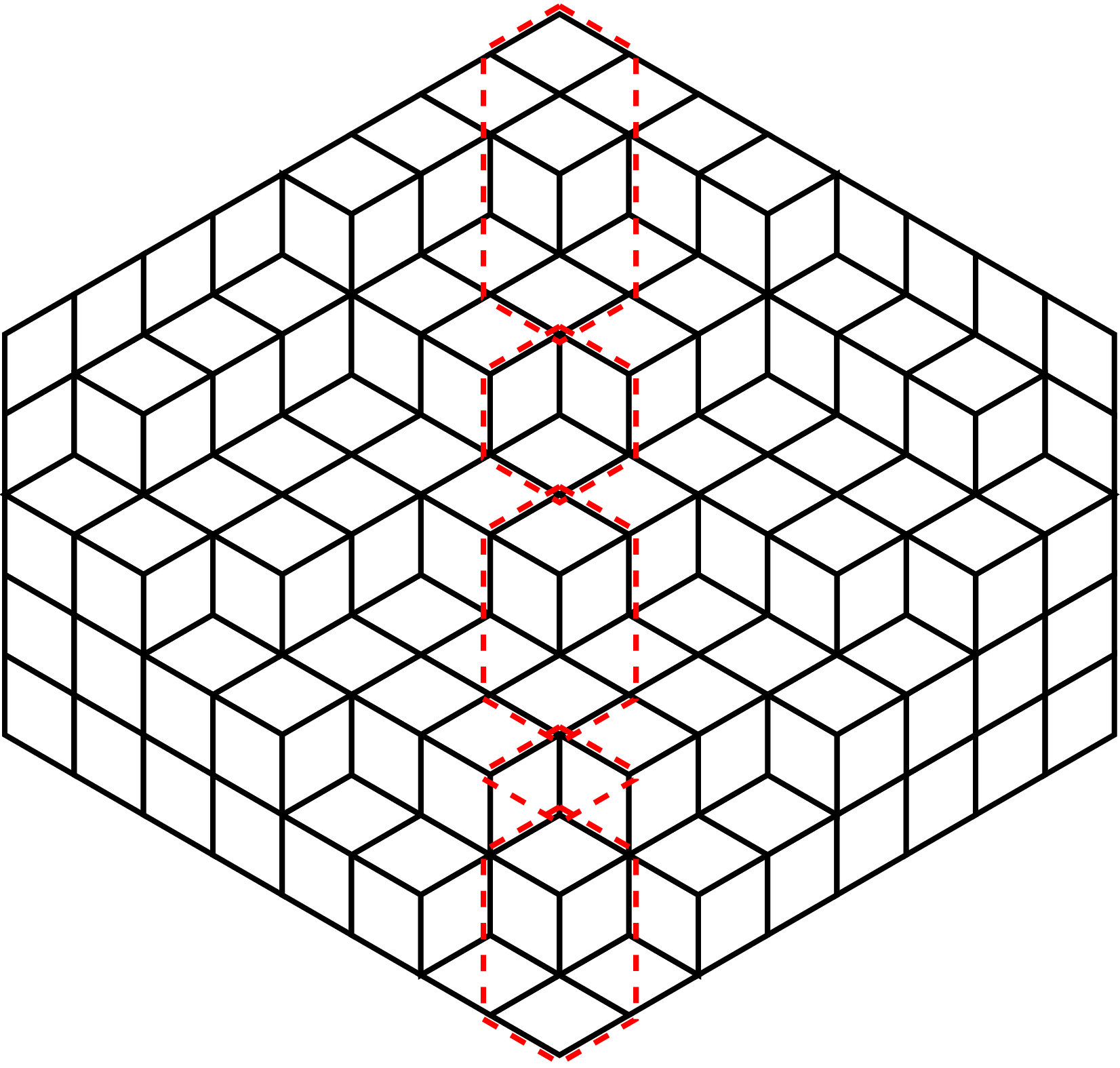}}
    \caption{(a) A symmetric hexagon $H(5,8,8)$ with a $2$-cell, a $0$-cell, a $2$-cell, a $1$-cell, and a $3$-cell marked (from bottom to top). (b) A $(2,0,2,1,3)$-block diagonally symmetric lozenge tiling of $H(5,8,8)$.}
\end{figure}
\begin{definition}\label{def.rblock}
    A lozenge tiling of $H(n,m,m)$ is \textit{$\mathbf{r}$-block diagonally symmetric} if
    \begin{itemize}
        \item the tiling is symmetric about $\ell$, and
        \item the $r_i$-cell contains exactly $r_i$ full horizontal lozenges for $i=1,\dots,n$.
    \end{itemize}
\end{definition}
We denote the number of $\mathbf{r}$-block diagonally symmetric lozenge tilings of $H(n,m,m)$ by $\M^{\r}(H(n,m,m))$. For example, a $(2,0,2,1,3)$-block diagonally symmetric lozenge tiling of $H(5,8,8)$ is given in Figure \ref{fig:blockHex2}. This definition generalizes the notion of off-diagonally symmetric lozenge tilings of $H(n,n,n)$ given in Definition \ref{def.offhex}; it is the special case of $\mathbf{r}$-block diagonally symmetric lozenge tilings where $r_1=\cdots=r_n=1$.

Consider the left half subregion of $H(n,m,m)$, which is a trapezoidal region with sides of length $n$, $m$, $m+n$, and $m$ in a clockwise order from the left, and denote it by $T(n,m)$ (see Figure \ref{fig:blockTrap2}; the boundary trapezoid is $T(5,8)$). We then label the $m+n$ left-pointing unit triangles along the right side of $T(n,m)$ by $1,\ldots,m+n$ from bottom to top. For any subset $P \subset[m+n]$ of cardinality $m$, let $T(n,m;P)$ denote the region obtained from $T(n,m)$ by deleting unit triangles labeled by elements in $P$. Throughout this paper, the elements in the set of labels are written in increasing order. 

One important observation is the following. Due to the symmetry condition, symmetric lozenge tilings of $H(n,m,m)$ are uniquely determined by the lozenges placed strictly left to $\ell$ (see Figure \ref{fig:blockTrap}). For an $(r_1,\ldots,r_n)$-block diagonally symmetric lozenge tiling of $H(n,m,m)$, lozenges that are strictly left to $\ell$ form a lozenge tilings of $T(n,m;P)$, where $P$ satisfies the following condition: 
\begin{itemize}
    \item $P$ contains $r_{1}$ elements among $\{1,\ldots,r_{1}+1\}=[S_{1}+1]$,
    \item $P$ contains $r_{2}$ elements among $\{r_{1}+2,\ldots,r_1+r_2+2\}=[S_{2}+2]\setminus[S_{1}+1]$,
    \item[] \begin{center} $\vdots$ \end{center}
    \item $P$ contains $r_{n}$ elements among $\{r_{1}+\cdots+r_{n-1}+n,\ldots,r_{1}+\cdots+r_{n}+n\}=[S_{n}+n] \setminus [S_{n-1}+n-1]$.
\end{itemize}
These conditions are equivalent to 
\begin{equation}\label{eq.condition}
    |P\cap([S_{k}+k]\setminus[S_{k-1}+k-1])|=r_{k},
\end{equation}
for each $k=1,\ldots,n$.

Conversely, given a lozenge tiling of $T(n,m;P)$ with the set $P$ subject to the condition \eqref{eq.condition}, we can uniquely determine an $\mathbf{r}$-block diagonally symmetric lozenge tiling of $H(n,m,m)$. It can be obtained by reflecting the lozenge tiling of $T(n,m;P)$ across the right boundary of $T(n,m)$ and filling each dent of $T(n,m;P)$ and its mirror image with a full lozenge. From this correspondence, one can see that the number of horizontal lozenges that cross $\ell$ equals the cardinality of $P$. This is because the region $T(n,m)$ consists of $m$ more left-pointing unit triangles than right-pointing ones, and $T(n,m;P)$ has no tilings unless $|P|=m$.
\begin{figure}[hbt!]
    \centering
    \subfigure[]{\label{fig:blockTrap1}\includegraphics[width=0.32\textwidth]{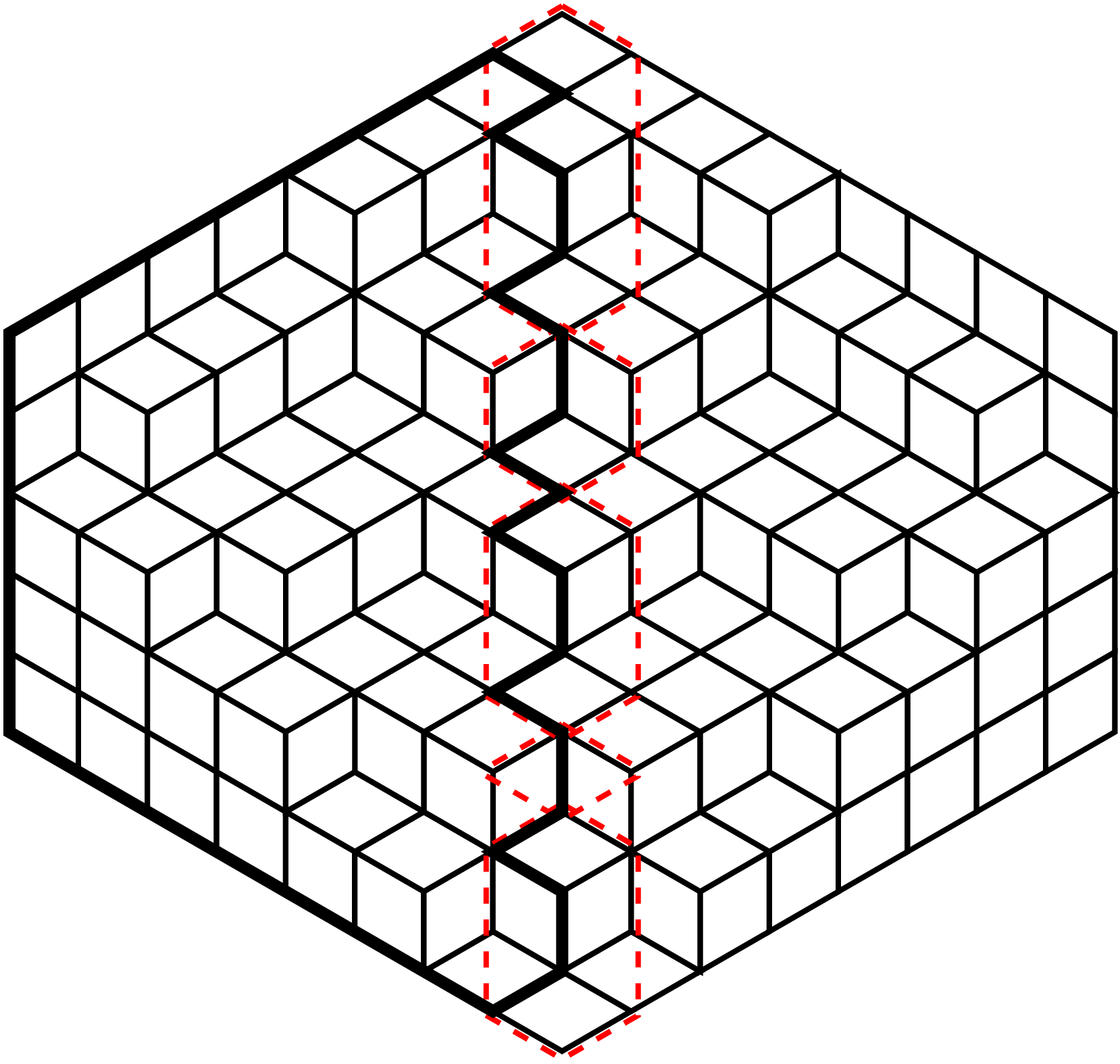}}
    \hspace{15mm}
    \subfigure[]{\label{fig:blockTrap2}\includegraphics[width=0.16\textwidth]{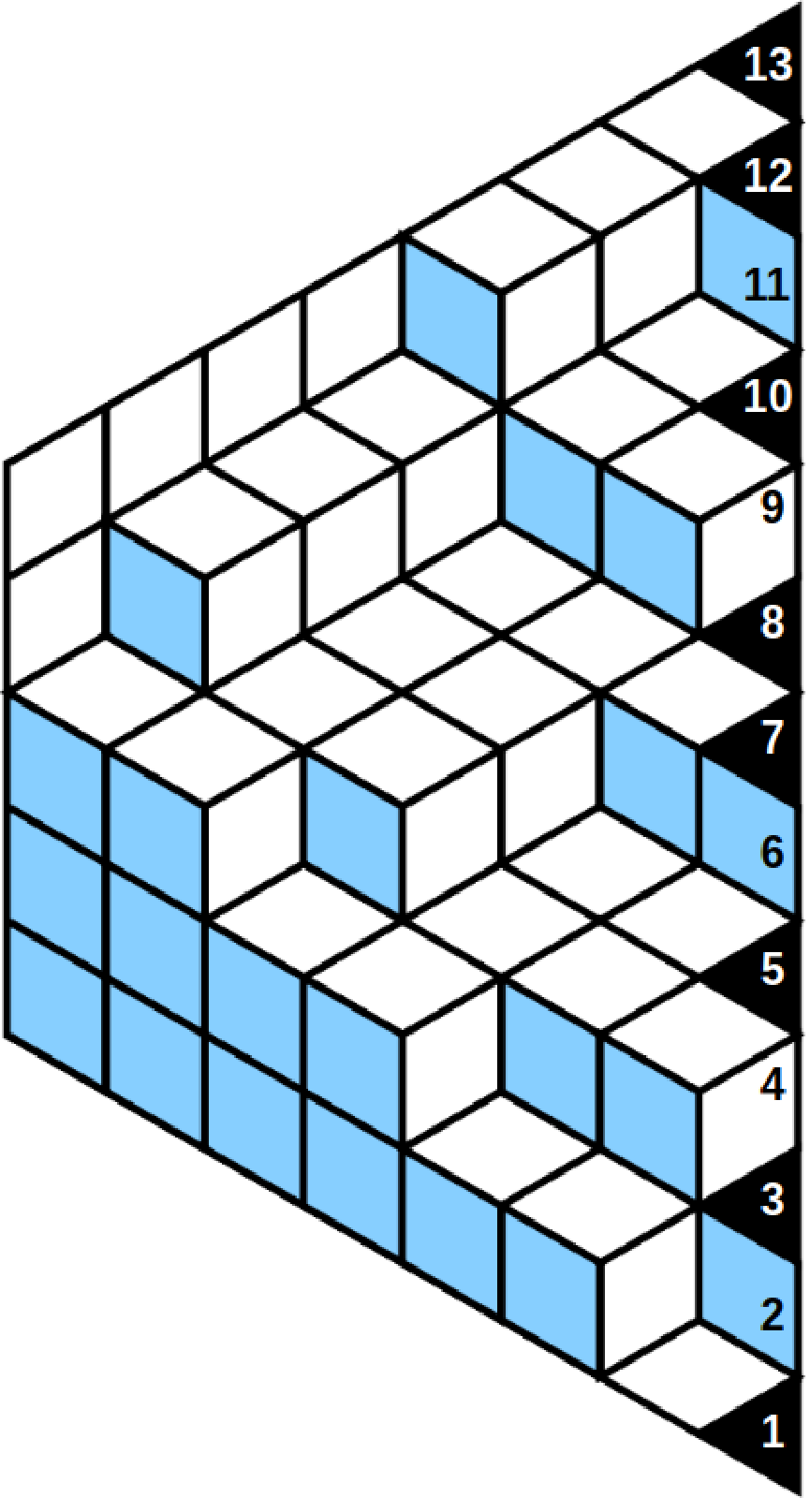}}
    \caption{(a) A $(2,0,2,1,3)$-block diagonally symmetric lozenge tiling of $H(5,8,8)$. Due to the symmetry, the tiling is uniquely determined by the lozenges inside the marked subregion. (b) The corresponding lozenge tiling of the region $T(5,8;P)$ where $P=\{1,3,5,7,8,10,12,13\}$. From right to left in each column, there are $3$, $3$, $3$, $2$, $3$, $2$, $4$, and $3$ (shaded) lozenges that are weighted by $x_1,x_2,\dots,x_8$, respectively. All the other lozenges are weighted by $1$.}\label{fig:blockTrap}
\end{figure}

Now, we describe how to assign a weight to lozenges contained in a lozenge tiling of $T(n,m;P)$. Lozenges are categorized into three types based on the orientation of their long diagonals: \textit{horizontal} lozenges, \textit{positive} lozenges, and \textit{negative} lozenges, according to the slope of their long diagonals. Among them, we assign a weight of $1$ to horizontal and positive lozenges. For negative lozenges, we assign weights $x_{k+1}$ to them, where the distance between the right side of such a lozenge and the right boundary of the trapezoid $T(n,m)$ is $\frac{\sqrt{3}}{2}k$; see Figure \ref{fig:blockTrap2}. Under this general weight assignment on lozenges, we write $\M_{\x}(T(n,m;P))$ to denote the tiling generating function of $T(n,m;P)$. 

Based on the above discussion, we define the generating function of $\r$-block diagonally symmetric lozenge tilings of $H(n,m,m)$ as
\begin{equation}\label{eq.xtrapezoidsum}
    \M_{\x}^{\r}\left( H(n,m,m) \right) := \sum_{P}\M_{\x}(T(n,m;P)),    
\end{equation}
where the sum runs over the sets $P$ described in \eqref{eq.condition}. In this paper, we also consider the \textit{$(q,t)$-weight} of a lozenge tiling of $T(n,m;P)$ by setting $x_k = q^{k-1}t$; this weight assignment is closely related to the volume of a pile of unit cubes. We use $\M_{q,t}^{\r}(H(n,m,m))$ and $\M_{q,t}^{\r}(T(n,m;P))$ to denote the $(q,t)$-generating functions of $\r$-block diagonally symmetric lozenge tilings of $H(n,m,m)$ and lozenge tilings of $T(n,m;P)$, respectively.
\begin{remark}
    The sum in \eqref{eq.xtrapezoidsum} was previously studied when it runs over different sets $P$.
    \begin{itemize}
        \item When the sum contains only one term and the set $P$ contains $m$ fixed elements: \eqref{eq.xtrapezoidsum} equals a certain Schur polynomial (see \eqref{eq.AFschur}). In particular, specializing all the weights $x_i=1$ yields the number of lozenge tilings of a trapezoidal region with $m$ unit triangles labeled by $P$ deleted, which was enumerated by Cohn, Larsen, and Propp \cite{CLP}.
         
        \item When the sum runs over all the $m$-subsets $P$ of $[m+n]$: \eqref{eq.xtrapezoidsum} corresponds to a sum of certain Schur polynomials. Equations \eqref{eq.MacMsscq1} and \eqref{eq.MacMssc} are specializations of such a sum: the latter corresponds to the number of symmetric plane partitions $|\{ \pi \in PP^n(m \times m )~|~ \mathsf{tr}(\pi) = \pi \}|$.
    \end{itemize}
    One can view the sum in \eqref{eq.xtrapezoidsum} as an interpolation of the two previously studied cases.
\end{remark}
Our first main theorem provides a simple product formula for the $(q,t)$-generating function of $\r$-block diagonally symmetric lozenge tilings of $H(n,m,m)$. 
\begin{theorem}\label{thm:1}
    For positive integers $m$ and $n$, let $\r=(r_1,\ldots,r_n)$ be an $n$-tuple of non-negative integers such that $r_1+\cdots+r_n=m$. The $(q,t)$-generating function of $\r$-block diagonally symmetric lozenge tilings of $H(n,m,m)$ is given by the following formula:
    \begin{equation}\label{eq:thm1.1}
        \M_{q,t}^{\r} (H(n,m,m)) = \prod_{i=1}^{m}(1+q^{i-1}t)\cdot q^{\alpha}\cdot t^{\beta}\cdot\frac{\prod_{1\leq i<j\leq n}[S_{j}-S_{i}+j-i]_{q}\cdot\prod_{i=1}^{n}[S_{n}+i-1]_{q}!}{\prod_{i=1}^{n}[S_{i}+i-1]_{q}![S_{n}-S_{i}+n-i]_{q}!},
    \end{equation}
    where $S_k = \sum_{i=1}^{k}r_i$, $\alpha=\sum_{i=1}^{n}\binom{S_{n}-S_{i}}{2}$, and $\beta=\sum_{i=1}^{n}(S_{n}-S_{i})$.
\end{theorem}

As a direct consequence of Theorem \ref{thm:1}, the number of $\mathbf{r}$-block diagonally symmetric lozenge tilings of $H(n,m,m)$ is obtained by taking $t=1$ and $q\rightarrow1$ in \eqref{eq:thm1.1}.
\begin{corollary}\label{cor:1}
    Let $m,n,\r$, and $S_k$ be the same as in Theorem \ref{thm:1}. The number of $\mathbf{r}$-block diagonally symmetric lozenge tilings of $H(n,m,m)$ is given by the following product formula:
    \begin{equation}\label{eq:cor1}
       \M^{\r} (H(n,m,m)) = 2^{m}\cdot\frac{\prod_{1\leq i<j\leq n}(S_{j}-S_{i}+j-i)\cdot\prod_{i=1}^{n}(S_{n}+i-1)!}{\prod_{i=1}^{n}(S_{i}+i-1)!(S_{n}-S_{i}+n-i)!}.
    \end{equation}
\end{corollary}
Remarkably, if we set $r_{1}=\cdots=r_{n}=1$, then $S_{i}=i$ for each $i=1,\dots,n$ and \eqref{eq:cor1} simplifies to
\begin{equation*}
    2^{n}\cdot\frac{\prod_{1\leq i<j\leq n}(2j-2i)\cdot\prod_{i=1}^{n}(n+i-1)!}{\prod_{i=1}^{n}(2i-1)!(2n-2i)!}=2^{n}\cdot2^{(n-1)n/2}=2^{n(n+1)/2}.
\end{equation*}
This proves that the number of off-diagonally symmetric lozenge tilings of a regular hexagon $H(n,n,n)$ is equinumerous to the number of domino tilings of $AD(n)$, which is mentioned earlier in Section \ref{sec:off-diag}.

On the other hand, if we set $r_1=\cdots=r_n=2$, then $S_i=2i$ for $i=1,\ldots,n$ and \eqref{eq:cor1} becomes
\begin{equation*}
    2^{2n}\cdot\frac{\prod_{1\leq i<j\leq n}(3j-3i)\cdot\prod_{i=1}^{n}(2n+i-1)!}{\prod_{i=1}^{n}(3i-1)!(3n-3i)!}=2^{2n}\cdot3^{\frac{n(n-1)}{2}}\cdot\frac{\prod_{i=1}^{n}(i-1)!(2n+i-1)!}{\prod_{i=1}^{n}(3i-1)!(3n-3i)!}=2^{2n}\cdot3^{\frac{n(n-1)}{2}}\cdot\prod_{k=0}^{n-1}\frac{(3k+1)!}{(n+k)!}.
\end{equation*}
Remarkably, $\prod_{k=0}^{n-1}\frac{(3k+1)!}{(n+k)!}$ is the number of alternating sign matrices (ASMs) of size $n$; see Remark \ref{rmk.asm} for a detailed discussion. It would be interesting to see if \eqref{eq:cor1} specializes to counting formulas (up to simple factors) for other combinatorial objects with different choices of $r_1,\ldots,r_n$; we leave this direction to be pursued by the interested reader.

\subsection{$\r$-block diagonally symmetric plane partitions}\label{sec:rblockpp}

We can translate the definition of $\mathbf{r}$-block diagonally symmetric lozenge tilings (in Definition \ref{def.rblock}) into certain symmetric plane partitions $\pi \in PP^n(m \times m)$ by imposing additional constraints on the diagonal entries $\pi_{i,i}$ for $1\leq i\leq m$. Let $\r=(r_1,\ldots,r_n)$ be an $n$-tuple of non-negative integers such that $r_1+\cdots+r_n=m$.
\begin{definition}\label{def.blocksympp}
    A plane partition $\pi \in PP^n(m \times m)$ is \textit{$\r$-block diagonally symmetric} if 
    \begin{itemize}
        \item $\mathsf{tr}(\pi) = \pi$ and 
        \item $\pi_{i,i} \in \{k-1,k\}$ if $i \in [S_n-S_{k-1}] \setminus [S_n-S_k]$, for each $k=1,\dots,n$.
    \end{itemize}   
\end{definition}
For example, the left picture in Figure \ref{fig:SPPcorrespondence} is a $(2,0,2,1,3)$-block diagonally symmetric plane partition in $PP^5(8 \times 8)$. One can easily check that it is invariant under taking the transpose. Furthermore, $S_1=2$, $S_2=2$, $S_3=4$, $S_4=5$, and $S_5=8$ in this case, and one can confirm that the diagonal entries satisfy the second condition in Definition \ref{def.blocksympp}:
\begin{itemize}
    \item $\pi_{i,i}\in\{0,1\}$ for $i\in [8-0]\setminus[8-2]=\{7,8\}$, 
    \item $\pi_{i,i}\in\{1,2\}$ for $i\in [8-2]\setminus[8-2]=\varnothing$, 
    \item $\pi_{i,i}\in\{2,3\}$ for $i\in [8-2]\setminus[8-4]=\{5,6\}$,
    \item $\pi_{i,i}\in\{3,4\}$ for $i\in [8-4]\setminus[8-5]=\{4\}$, and
    \item $\pi_{i,i}\in\{4,5\}$ for $i\in [8-5]\setminus[8-8]=\{1,2,3\}$.
\end{itemize}

\begin{figure}[hbt!]
    \centering 
    \includegraphics[width=0.55\textwidth]{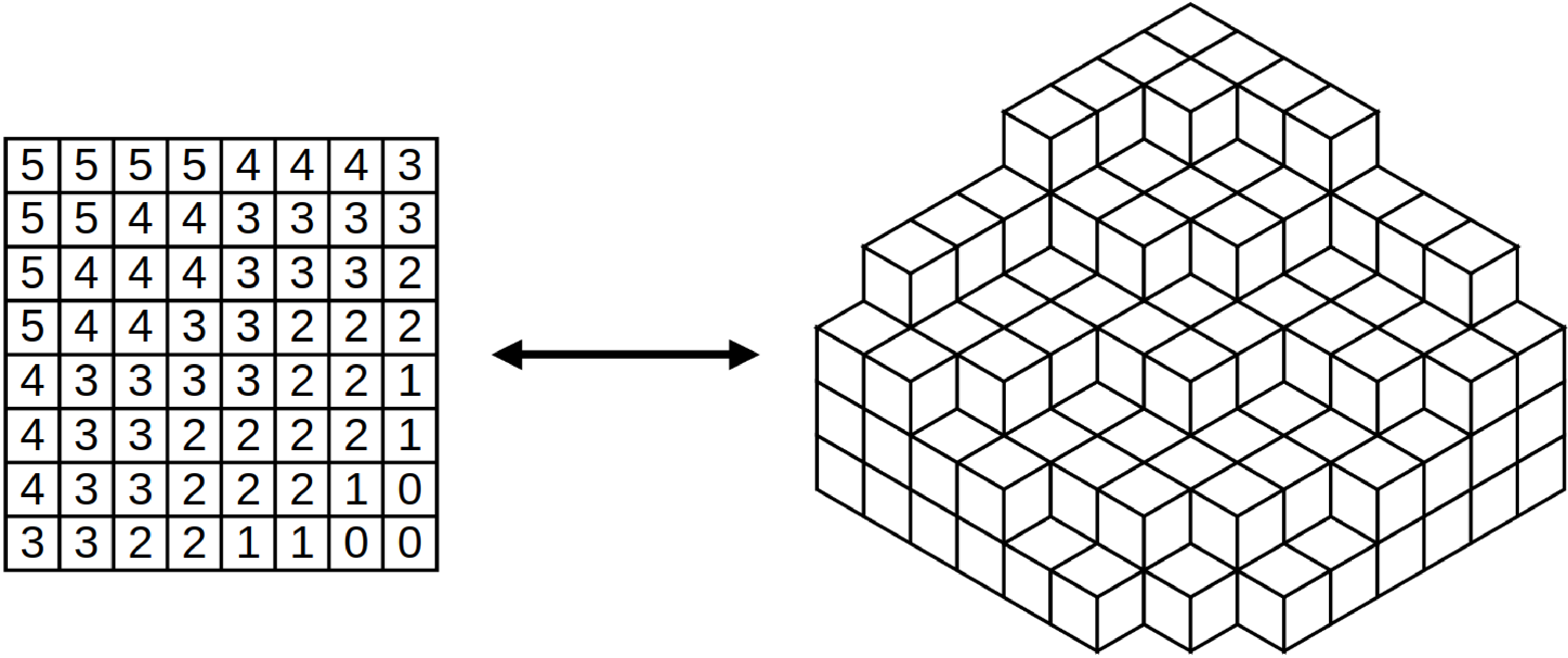}
    \caption{A $(2,0,2,1,3)$-block diagonally symmetric plane partition in $PP^5(8 \times 8)$ and its 3D realization as a pile of unit cubes. Since there are $23$ unit cubes on the main diagonal and $187$ unit cubes in total, this plane partition is weighted by $q^{187-23}t^{23}=q^{164}t^{23}$.}
    \label{fig:SPPcorrespondence}
\end{figure}

Let $PP^n_{\r}(m \times m)$ be the set of $\r$-block diagonally symmetric plane partitions in $PP^n(m \times m)$. See Figure \ref{fig:SPPcorrespondence} for an example. Via the bijection of David and Tomei \cite{DT} mentioned in Section \ref{sec:intro}, it is not hard to see that $PP^n_{\r}(m \times m)$ is in bijection with the set of $\r$-block diagonally symmetric lozenge tilings of $H(n,m,m)$; the second condition in Definition \ref{def.blocksympp} comes from the condition that the $r_i$-cell contains exactly $r_i$ full horizontal lozenges in Definition \ref{def.rblock}. For each $\r$-block diagonally symmetric plane partition $\pi \in PP^n_{\r}(m \times m)$, we associate a weight $q^{|\pi|_n}t^{|\pi|_d}$ to it, where $|\pi|_d=\sum_{i}\pi_{i,i}$ is the sum of the diagonal entries of $\pi$ and $|\pi|_n=\sum_{i\neq j}\pi_{i,j}$ is the sum of non-diagonal entries of $\pi$. 

\begin{figure}[hbt!]
    \centering
    \subfigure[]{\label{fig:hook1}\includegraphics[width=0.23\textwidth]{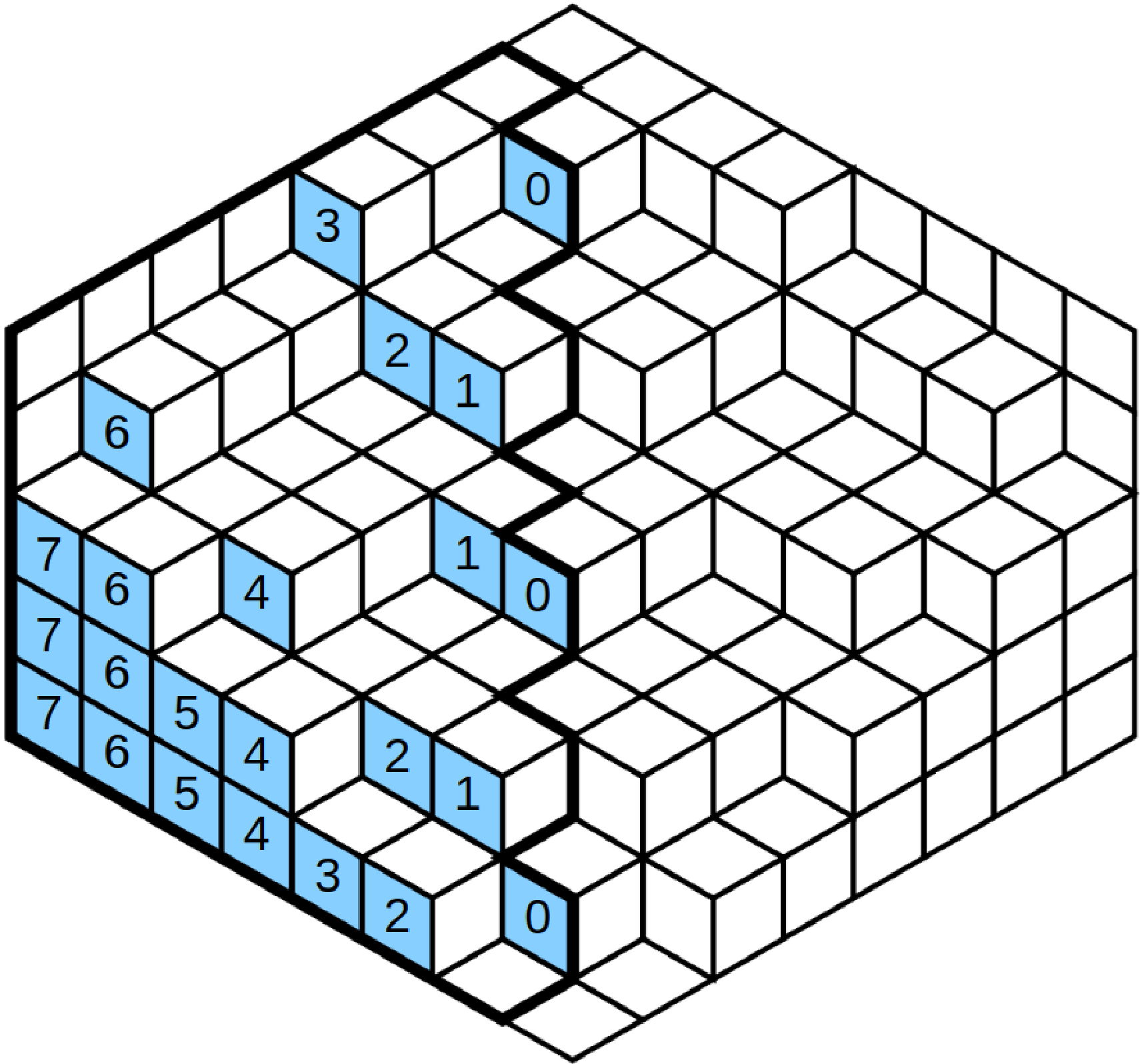}}
    \hspace{10mm}
    \subfigure[]{\label{fig:hook2}\includegraphics[width=0.23\textwidth]{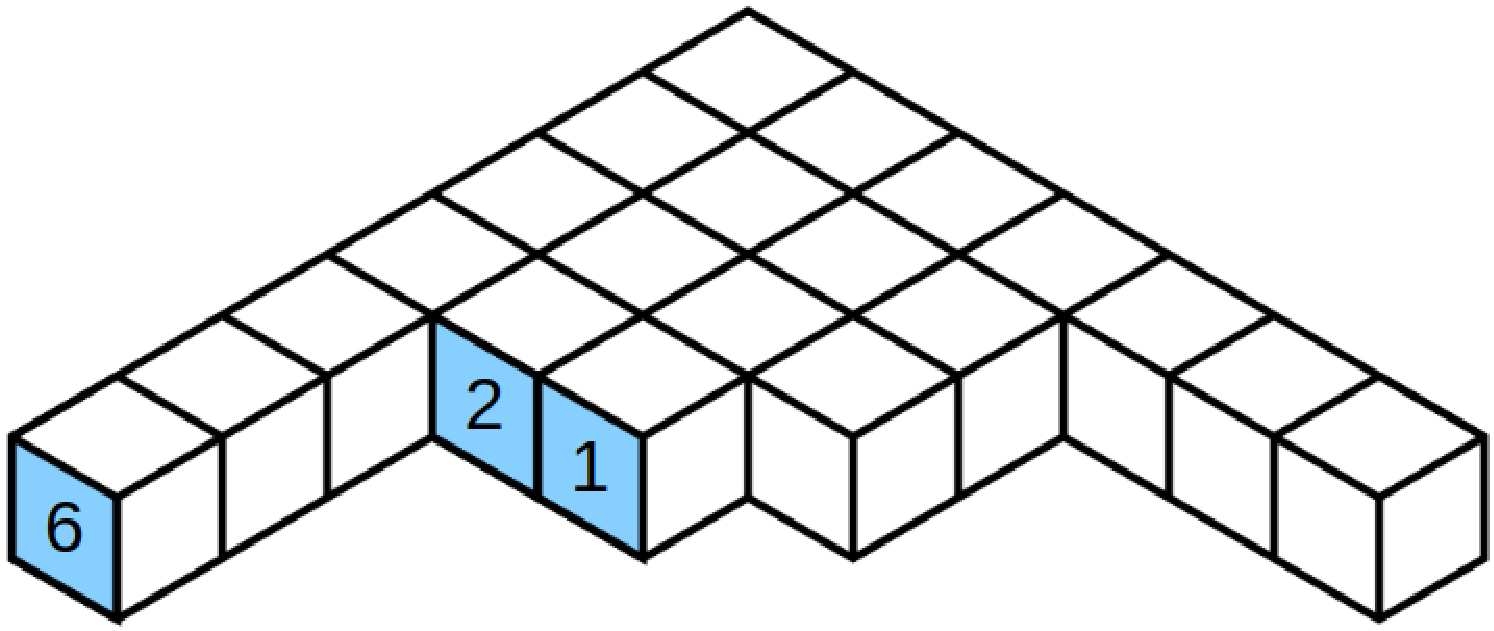}}
    \hspace{10mm}
    \subfigure[]{\label{fig:hook3}\includegraphics[width=0.23\textwidth]{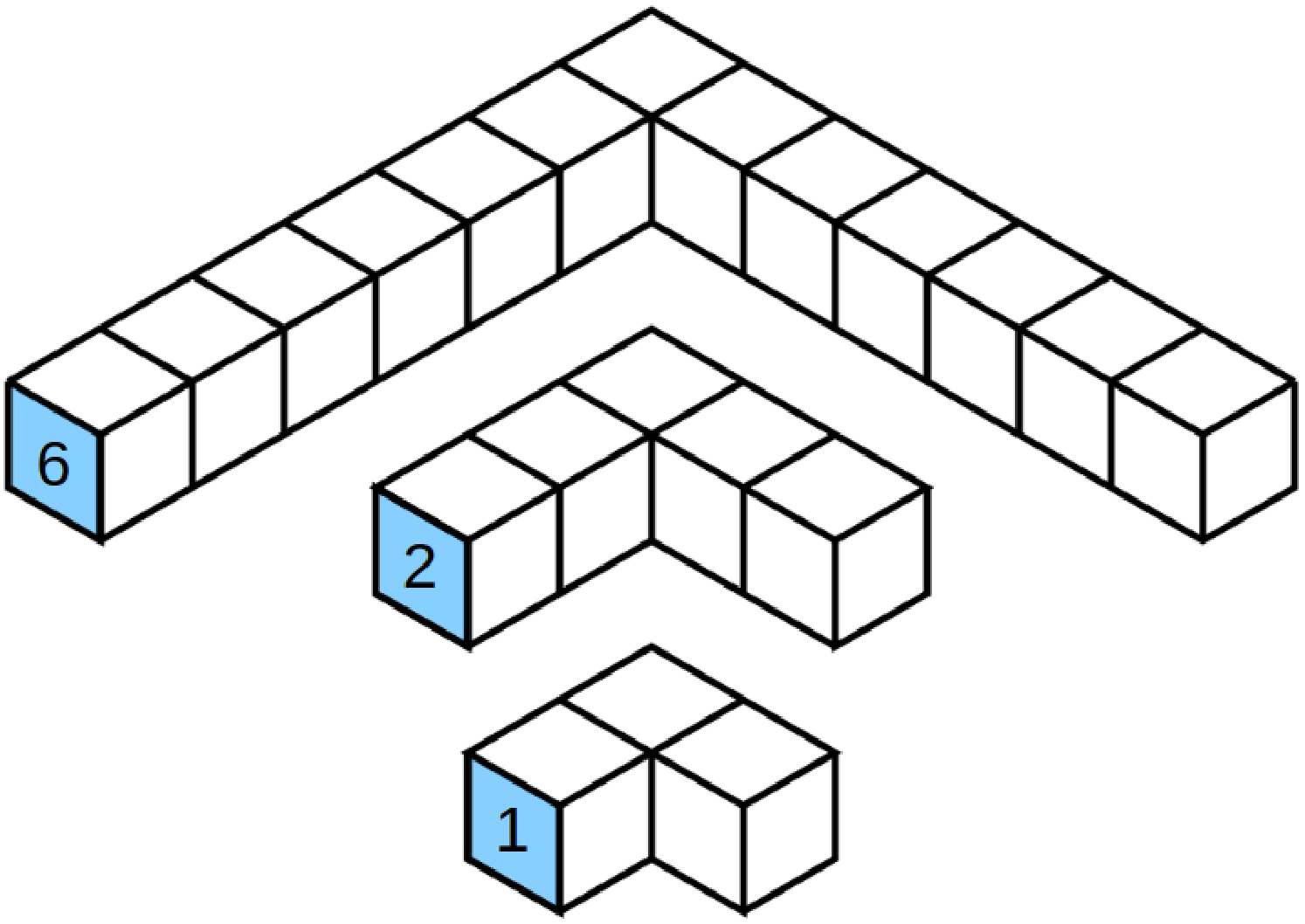}}
    \caption{(a) A $(2,0,2,1,3)$-block diagonally symmetric lozenge tiling of $H(5,8,8)$ in Figure \ref{fig:blockTrap1}. A shaded lozenge labeled $k$ is weighted by $q^kt$. (b) A collection of unit cubes in the fourth level in Figure \ref{fig:hook1}, when we view it as a (symmetric) plane partition. (c) A decomposition of unit cubes in Figure \ref{fig:hook2} into three symmetric hooks.}
    \label{fig:hook}
\end{figure}

Now, we explain how the weight of $\r$-block diagonally symmetric plane partitions behaves under David and Tomei's bijection. To do that, we view such a symmetric lozenge tiling as a pile of unit cubes and decompose it as follows (see Figure \ref{fig:hook}). We first partition these unit cubes according to their level and then decompose the collection of unit cubes in the same level into symmetric hooks. Figure \ref{fig:hook2} shows all the unit cubes in the fourth level in Figure \ref{fig:hook1}; their symmetric hook decomposition is given in Figure \ref{fig:hook3}.

One can see that the leftmost unit cube in each symmetric hook corresponds to a nontrivially weighted lozenge, and our weight assignment carries the following geometric information: if this lozenge is weighted by $q^kt$, then the corresponding symmetric hook consists of $(2k+1)$ unit cubes (see Figure \ref{fig:hook3} again). In this symmetric hook, we can interpret the exponent of $q^k$ as half of the number of unit cubes in this symmetric hook that are not on the diagonal, and that of $t=t^1$ as the unit cube on the diagonal. 

Thus, the $(q,t)$-generating function of $\r$-block diagonally symmetric lozenge tilings of $H(n,m,m)$ (as shown in Theorem \ref{thm:1}) gives the following generating function of $\r$-block diagonally symmetric plane partitions $PP^n_{\r}(m\times m)$:
\begin{equation}\label{eq:relation}
    \sum_{\pi \in PP^n_{\r}(m \times m)} q^{\frac{1}{2}|\pi|_n}t^{|\pi|_d}.
\end{equation}
Therefore, replacing $q$ by $q^2$ in \eqref{eq:thm1.1}, we derive the following $(q,t)$-generating function of $\r$-block diagonally symmetric plane partitions $PP^n_{\r}(m \times m)$. 
\begin{theorem}\label{thm:1.5}
    The $(q,t)$-generating function of $\r$-block diagonally symmetric plane partitions $PP^n_{\r}(m \times m)$ is given by
    \begin{equation}\label{eq:cor2}
        \sum_{\pi \in PP^n_{\r}(m \times m)} q^{|\pi|_n}t^{|\pi|_d}=\prod_{i=1}^{m}(1+q^{2i-2}t)\cdot q^{2\alpha}\cdot t^{\beta}\cdot\frac{\prod_{1\leq i<j\leq n}[S_{j}-S_{i}+j-i]_{q^2}\cdot\prod_{i=1}^{n}[S_{n}+i-1]_{q^2}!}{\prod_{i=1}^{n}[S_{i}+i-1]_{q^2}![S_{n}-S_{i}+n-i]_{q^2}!},
    \end{equation}
    where $S_k = \sum_{i=1}^{k}r_i$, $\alpha=\sum_{i=1}^{n}\binom{S_{n}-S_{i}}{2}$, and $\beta=\sum_{i=1}^{n}(S_{n}-S_{i})$.
\end{theorem}

If we set $t=q$ in \eqref{eq:cor2}, the weight of each plane partition $\pi$ becomes $q^{|\pi|_n}t^{|\pi|_d}=q^{|\pi|_n+|\pi|_d}=q^{|\pi|}$, which coincides with the volume of $\pi$. Thus, under this specialization, we obtain the volume generating function of $\r$-block diagonally symmetric plane partitions $PP^n_{\r}(m \times m)$.
\begin{corollary}\label{cor:3}
    The volume generating function of $\r$-block diagonally symmetric plane partitions $PP^n_{\r}(m \times m)$ is given by
    \begin{equation}\label{eq:cor3}
        \sum_{\pi \in PP^n_{\r}(m \times m)} q^{|\pi|} = \prod_{i=1}^{m}(1+q^{2i-1})\cdot q^{2\alpha+\beta}\cdot\frac{\prod_{1\leq i<j\leq n}[S_{j}-S_{i}+j-i]_{q^2}\cdot\prod_{i=1}^{n}[S_{n}+i-1]_{q^2}!}{\prod_{i=1}^{n}[S_{i}+i-1]_{q^2}![S_{n}-S_{i}+n-i]_{q^2}!},
    \end{equation}
    where $S_k = \sum_{i=1}^{k}r_i$, $\alpha=\sum_{i=1}^{n}\binom{S_{n}-S_{i}}{2}$, and $\beta=\sum_{i=1}^{n}(S_{n}-S_{i})$.
\end{corollary}

\subsection{$(\r,\r^{\prime})$-block diagonally symmetric lozenge tilings}\label{sec:rrblock}

We generalize the notion of the $\r$-block diagonal symmetry class. For positive integers $m,n,l$, let $\r=(r_1,\dots,r_n)$ and $\r^{\prime}=(r_1^{\prime},\dots,r_n^{\prime})$ be two $n$-tuples of non-negative integers such that $|\r| = m+l$ and $|\r^{\prime}| = l$. We consider the region obtained by gluing the left and right sides of the hexagon $H(n+l,m,m)$ and embedding it into a cylinder. We denote this identified side by $\ell^{\prime}$. We then place $r_i$-cells on the vertical symmetry axis $\ell$ of $H(n+l,m,m)$ as before, and similarly place $r_i^{\prime}$-cells on $\ell^{\prime}$ such that these cells are non-overlapping and they are bisected by $\ell^{\prime}$. 
\begin{definition}\label{def.rrblock}
    We say a lozenge tiling of $H(n+l,m,m)$ is \textit{$(\r,\r^{\prime})$-block diagonally symmetric} if 
\begin{itemize}
    \item the tiling is $\r$-block diagonally symmetric, and
    \item the $r_i^{\prime}$-cell contains exactly $r_i^{\prime}$ full horizontal lozenges for $i=1,\dots,n$.
\end{itemize}
\end{definition}
See Figure \ref{fig:cylinder1} for an example. We note that if $l = 0$ (i.e. $\r^{\prime} = (0,\dots,0)$), then in an  $(\r,\r^{\prime})$-block diagonally symmetric lozenge tiling, there is no horizontal lozenge cross $\ell^{\prime}$. Thus, it reduces to an $\r$-block diagonally symmetric lozenge tiling.

\begin{figure}[hbt!]
    \centering
    \subfigure[]{\label{fig:cylinder1}\includegraphics[width=0.18\textwidth]{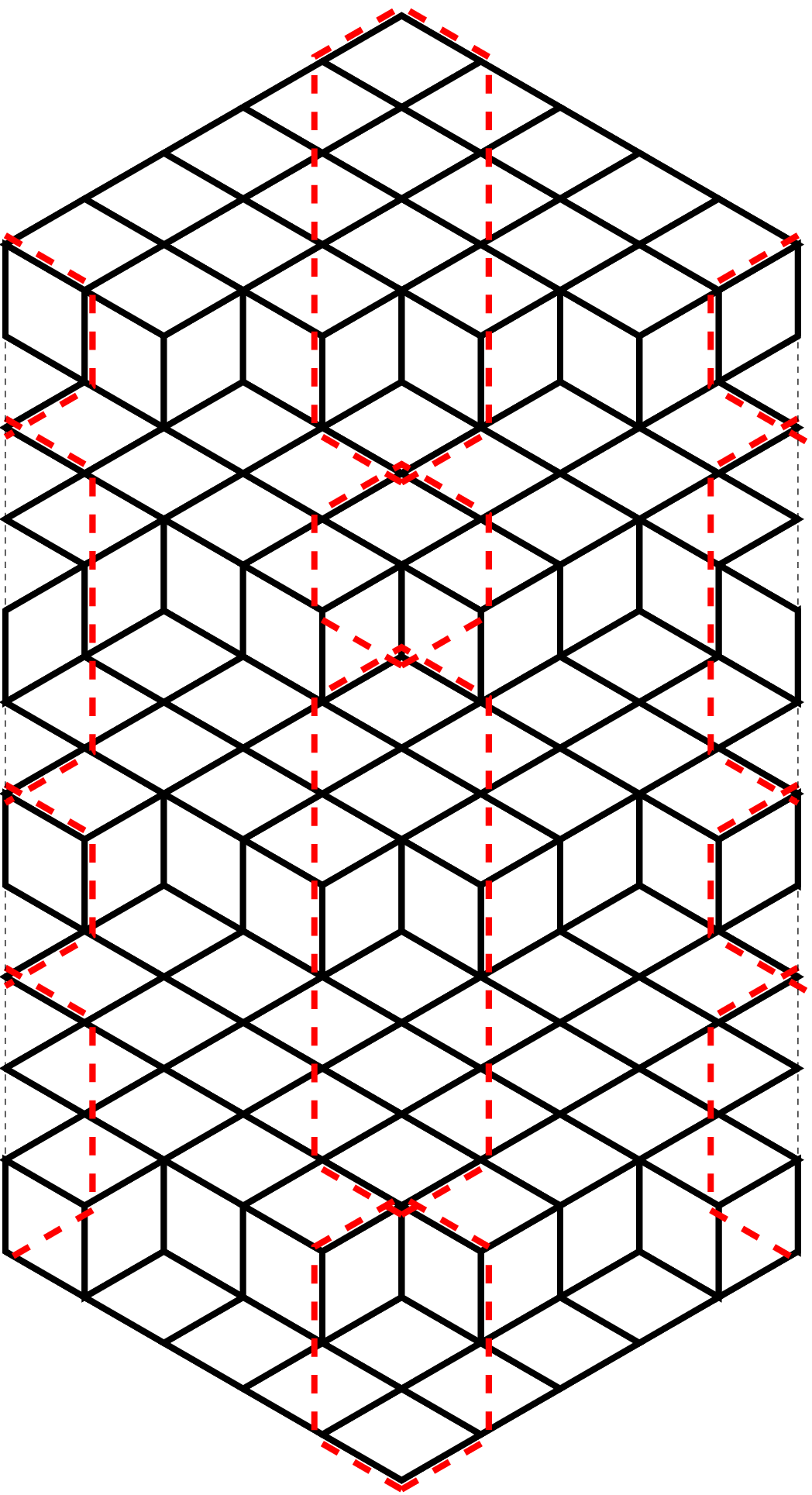}}
    \hspace{20mm}
    \subfigure[]{\label{fig:cylinder2}\includegraphics[width=0.09\textwidth]{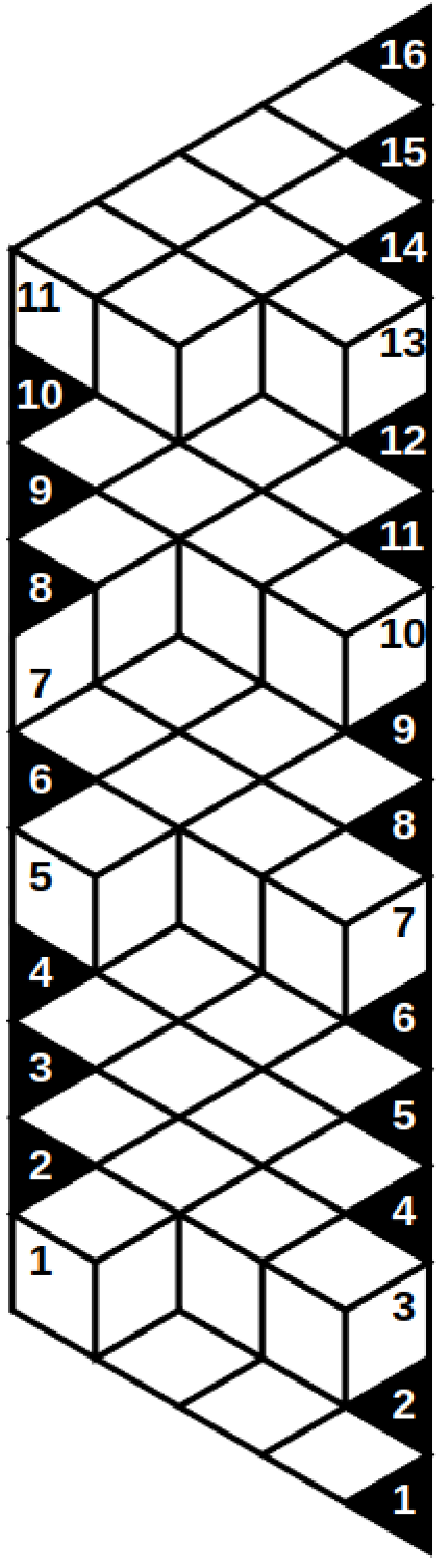}}
    \caption{(a) An $(\r,\r^{\prime})$-block diagonally symmetric lozenge tiling of $H(11,5,5)$ with $\r=(2,5,1,4)$ and $\r^{\prime}=(2,1,3,1)$. (b) The corresponding lozenge tiling of the region $T(11,5;P,P^{\prime})$ with $P=\{1,2,4,5,6,8,9,11,12,14,15,16\}$ and $P^{\prime}=\{2,3,4,6,8,9,10\}$.}
\end{figure}

Moreover, an $(\r,\r^{\prime})$-block diagonally symmetric lozenge tiling of $H(n+l,m,m)$ can be identified with the trapezoidal region $T(n+l,m)$ with dents on the left and right sides, which is described as follows. From bottom to top, we label $n+l$ right-pointing unit triangles on the left side of $T(n+l,m)$ by $1,\ldots,n+l$ and label $m+n+l$ left-pointing unit triangles on the right side of $T(n+l,m)$ by $1,\ldots,m+n+l$. For any set $P \subset[m+n+l]$ with cardinality $m+l$ and $P^{\prime} \subset[n+l]$ with cardinality $l$, we define $T(n+l,m;P,P^{\prime})$ to be the region obtained from $T(n+l,m)$ by deleting the left-pointing unit triangles labeled by the elements of $P$ and the right-pointing unit triangles labeled by the elements of $P^{\prime}$. See Figures \ref{fig:cylinder1} and \ref{fig:cylinder2}.

We use the same weight assignment of lozenges described in Section \ref{sec:rblocklozenge}. The generating function of  $(\r,\r^{\prime})$-block diagonally symmetric lozenge tilings of $H(n+l,m,m)$ is defined as follows, which is analogous to that of $\M^{\r}_{\x}\left( H(n+l,m,m) \right)$ in \eqref{eq.xtrapezoidsum}:
\begin{equation}\label{eq.traptwodentssum}
    \M^{\r,\r^{\prime}}_{\x}\left( H(n+l,m,m) \right) \coloneqq \sum_{P,P^{\prime}} \M_{\x}\left( T(n+l,m;P,P^{\prime}) \right),
\end{equation}
where the sum runs over all sets $P \subset [m+n+l]$ satisfying the condition \eqref{eq.condition} and $P^{\prime} \subset [n+l]$ satisfying the following similar condition: for each $k=1,\ldots,n$,
\begin{equation}\label{eq.condition2}
    |P^{\prime}\cap([S^{\prime}_k+k]\setminus[S^{\prime}_{k-1}+k-1])|=r^{\prime}_k,
\end{equation}
where $S^{\prime}_k = \sum_{i=1}^{k}r^{\prime}_i$ denotes the partial sum of $(r_1^{\prime},\dots,r_n^{\prime})$ for each $k=1,\ldots,n$, with the convention $S^{\prime}_{0}=0$.

Our data indicates that when all variables are set to $1$, $\M^{\r,\r^{\prime}} (H(n+l,m,m))$ generally contains large prime factors, suggesting that a simple product formula is unlikely to exist. Nevertheless, we discover an identity regarding the signed enumeration of the $(\r,\r^{\prime})$-block diagonally symmetric lozenge tilings of $H(n+l,m,m)$. Before stating our second main result, we introduce the following notions.

Define two special sets as
\begin{equation}\label{eq.Rminmax}
    P_{\mathsf{min}} = [m+n+l] \setminus \{S_k+k|k=1,\dots,n\} \text{\quad and\quad} P^{\prime}_{\mathsf{max}} = [n+l] \setminus \{S^{\prime}_{k-1}+k|k=1,\dots,n\}.
\end{equation}
$P_{\mathsf{min}}$ is the set of minimal labels among all the sets $P$ satisfying the condition \eqref{eq.condition} while $P^{\prime}_{\mathsf{max}}$ is the set of maximal labels among all the sets $P^{\prime}$ satisfying the condition \eqref{eq.condition2}. Given any set $P \subset [m+n+l]$ satisfying \eqref{eq.condition} and $P^{\prime} \subset [n+l]$ satisfying \eqref{eq.condition2}, we define
\begin{equation}\label{eq.difference}
    d(P) = \sum_{x \in P} x - \sum_{x \in P_{\mathsf{min}}} x \text{\quad and \quad} d^{\prime}(P^{\prime}) = \sum_{x \in P^{\prime}_{\mathsf{max}}} x - \sum_{x \in P^{\prime}} x.
\end{equation}
These two quantities measure the total difference between the labels in $P$ and $P_{\mathsf{min}}$, as well as the labels in $P^{\prime}$ and $P^{\prime}_{\mathsf{max}}$, respectively.

Finally, the signed enumeration of $(\r,\r^{\prime})$-block diagonally symmetric lozenge tilings of $H(n+l,m,m)$ based on \eqref{eq.traptwodentssum} is defined to be
\begin{equation}\label{eq.trapetwodentssign}
    \M^{\r,\r^{\prime}}_{\x,\sgn}\left( H(n+l,m,m) \right) := \sum_{P,P^{\prime}} (-1)^{d^{\prime}(P^{\prime})}\M_{\x}\left( T(n+l,m;P,P^{\prime}) \right),
\end{equation}
where $P$ and $P^{\prime}$ are subject to the conditions given in \eqref{eq.traptwodentssum}. Now, we are ready to state the second main result.
\begin{theorem}\label{thm:2}
    For positive integers $m,n,l$ with $l \leq n$, let $\r=(r_1,\ldots,r_n)$ and $\r^{\prime} = (r^{\prime}_1,\dots,r^{\prime}_n)$ be two $n$-tuples of non-negative integers such that $|\r|=m+l$ and $|\r^{\prime}| = l$ with $r^{\prime}_i=1$ for $i=1,\dots,l$ and $r^{\prime}_j=0$ for $j=l+1,\dots,n$. Then the signed enumeration of $(\r,\r^{\prime})$-block diagonally symmetric lozenge tilings of $H(n+l,m,m)$ satisfies the following identity:
    \begin{equation}\label{eq.thm2}
         \M^{\r,\r^{\prime}}_{\x,\sgn}\left( H(n+l,m,m) \right) = \prod_{i=1}^{m}(1+x_i) \cdot \M_{\x}(T(n+l,m;P_{\mathsf{min}},P^{\prime}_{\mathsf{max}})),
    \end{equation}
    where $P_{\mathsf{min}}$ and $P^{\prime}_{\mathsf{max}}$ are defined in \eqref{eq.Rminmax}.
\end{theorem}

\begin{remark}
    In Theorem \ref{thm:2}, the condition on $\mathbf{r}'$ translates to the requirement that $P'$, the set of labels of deleted (right-pointing) unit triangles on the left side, contains exactly one element from each pair $\{2k-1, 2k\}$ for $k \in [l]$, as well as the entire set $[n+l] \setminus [2l]$. This structural property is essential; it allows us to express the signed enumeration of $(\r,\r^{\prime})$-block diagonally symmetric lozenge tilings in terms of Schur polynomials indexed by certain partitions, which is one of the key ingredients in our proof presented in \ref{sec:Thm2.10proof}.
\end{remark}

\section{The method of non-intersecting lattice paths}\label{sec:latticepath}

The method of non-intersecting lattice paths has been widely applied to enumerate lozenge or domino tilings of a region; see \cite[Section 3.1]{Propp15} and references therein. The idea is to establish a bijection between the set of tilings of a region and the set of families of non-intersecting lattice paths contained in this region. The way to enumerate families of non-intersecting lattice paths depends on the endpoints of these paths. 

If the set of starting and ending points are fixed and they are compatible\footnote{The set of starting points $U=\{u_1,\dots,u_n\}$ and ending points $V=\{v_1,\dots,v_n\}$ are said to be \textit{compatible} if the only way to form families of non-intersecting paths connecting $U$ with $V$ is from $u_i$ to $v_i$ for all $i$.}, then one can apply the Lindstr\"om--Gessel--Viennot (LGV) theorem to enumerate them. On the other hand, if the set of starting and ending points is fixed but they are not compatible (this usually appears when considering tilings of a region with holes), then one can apply the extension of the LGV theorem due to the second author \cite{Lee22}. For (diagonally) symmetric tilings, such tilings typically correspond to families of non-intersecting lattice paths with fixed starting points, while the ending points may vary over a specified set. In this case, Stembridge's generalization \cite{Stem90} provides a Pfaffian formula for these families (see \cite{Ciucu15} for an example).

The above methods do not apply directly to our $\r$-block diagonally symmetric lozenge tilings due to the block diagonally symmetric conditions stated in Definition \ref{def.rblock}. In this section, we review how to obtain a family of non-intersecting lattice paths from a lozenge tiling of $T(n,m;P)$, and then prove Theorem \ref{thm:1} by slightly modifying its corresponding lattice paths.

We present two lemmas that are used in our proof of Theorem \ref{thm:1}. Consider the integer lattice $\mathbb{Z}^2$ with horizontal (resp., vertical) lattice lines oriented toward east (resp., south). We assign a weight $q^{k}t$ to all the vertical unit segments in the strip bounded by two lines $y=x+k$ and $y=x+(k+1)$, for each integer $k$, and assign a weight $1$ to all the horizontal unit segments. For any two points $(x,y)$ and $(z,w)$ on $\mathbb{Z}^2$, let $\wt_{q,t}((x,y)\rightarrow(z,w))$ be the sum of weights\footnote{The weight of a given lattice path is the product of the weights of all unit segments in the path.} of all the lattice paths going from $(x,y)$ to $(z,w)$. The first lemma states that for any integers $a$ and $b$, the total weight of lattice paths joining $(a,b)$ and $(0,0)$ is given by a $(q,t)$-analogue of the binomial coefficient.
\begin{lemma}\label{lem:3.1}
    For integers $a$ and $b$, 
    \begin{equation}\label{eq:latticepathweight}
        \wt_{q,t}((a,b)\rightarrow(0,0))=q^{\frac{(b-1)b}{2}}t^{b}\begin{bmatrix} b-a\\b \end{bmatrix}_q.
    \end{equation}
\end{lemma}
\begin{proof}
    Since each step on $\mathbb{Z}^2$ can only move toward east or south, if $a>0$, $b<0$, or $b<a$, then there is no lattice path going from $(a,b)$ to $(0,0)$, and thus $\wt_{q,t}((a,b)\rightarrow(0,0))=0$. In these cases, \eqref{eq:latticepathweight} holds due to $\begin{bmatrix} b-a\\b \end{bmatrix}_q=0$. It suffices to verify \eqref{eq:latticepathweight} when $a\leq0$ and $b\geq0$.

    When $a=0$, there is only one (vertical) lattice path whose weight is $\prod_{i=0}^{b-1}q^it=q^{\frac{(b-1)b}{2}}t^b$. Similarly, when $b=0$, there is only one (horizontal) lattice path whose weight is $1$. In both cases, their weights agree with the right-hand side of \eqref{eq:latticepathweight}.

    Observe that the set of lattice paths going from $(a,b)$ to $(0,0)$ can be partitioned according to their first step. If the first step is horizontal, this horizontal edge has a weight of $1$ and the rest of the path is determined by the path connecting $(a+1,b)$ to $(0,0)$. On the other hand, if the first step is vertical, this vertical edge has a weight $q^{b-a-1}t$ and the rest of the path is determined by the path connecting $(a,b-1)$ to $(0,0)$. We obtain the following recurrence relation:
    \begin{equation}\label{eq:recurrence}
        \wt_{q,t}((a,b)\rightarrow(0,0))=\wt_{q,t}((a+1,b)\rightarrow(0,0))+q^{b-a-1}t\wt_{q,t}((a,b-1)\rightarrow(0,0)).
    \end{equation}
    Then \eqref{eq:latticepathweight} follows from the induction on the quantity $b-a$. We omit the computation here since this can be easily done using the definition of the $q$-binomial coefficient.
\end{proof}

Our second lemma is a determinant formula due to Krattenthaler. 
\begin{lemma}{\cite[Equation (3.12)]{Krattenthaler}}\label{lem:3.2}
    Given a non-negative integer $n$. Let $L_1,L_2,\ldots,L_n$ and $M$ be indeterminates. Then there holds
    \begin{equation*}
        \det\left(q^{jL_i}\begin{bmatrix} M\\L_i+j \end{bmatrix}_q\right)_{1\leq i,j\leq n}=q^{(\sum_{i=1}^{n}iL_i)} \cdot \frac{\prod_{1\leq i<j\leq n}[L_{i}-L_{j}]_q}{\prod_{i=1}^{n}[L_{i}+n]_q!}\cdot\frac{\prod_{i=1}^{n}[M+i-1]_q!}{\prod_{i=1}^{n}[M-L_{i}-1]_q!}.
    \end{equation*}
\end{lemma}

For positive integers $m,n$, let $P$ be a fixed $m$-subset of $[m+n]$. We present below how to encode a lozenge tiling of $T(n,m;P)$ as families of $n$ non-intersecting lattice paths on $\mathbb{Z}^2$. Let $U_T$ and $V_T$ be the sets of midpoints of each unit segment on the left side and the right side of $T(n,m)$, respectively. Thus, $|U_T|=n$ and $|V_T|=m+n$ and we assume that the elements of $U_T$ and $V_T$ are ordered from bottom to top.

Given a lozenge tiling of $T(n,m;P)$, we decorate each positive and negative lozenge by a path step joining the midpoints of two vertical sides of that lozenge. Note that we do not decorate horizontal lozenges. Regarding each of these path steps as oriented from left to right, the union of all these path steps forms an $n$-tuple of non-intersecting lattice paths connecting $U_T$ with an $n$-subset of $V_T$, which depends on $P$. Furthermore, we transform them into $n$ non-intersecting lattice paths on $\mathbb{Z}^2$; see Figures \ref{fig:path1} and \ref{fig:path2} for an example. 

Specifically, we denote the complement of $P$ as $[m+n]\setminus P=\{t_1,\ldots,t_{n}\}$, where the elements are written in increasing order. On $\mathbb{Z}^2$, let $U=\{u_1,\dots,u_n\}$ and $V(P)=\{v_1,\dots,v_n\}$ be two collections of points, where $u_{k}=(k,m+k)$ and $v_{k}=(t_k,t_k)$, for $k=1,\ldots n$. We write $\mathscr{P}(U,V(P))$ for the set of $n$-tuples of non-intersecting lattice paths on $\mathbb{Z}^2$ going from $U$ to $V(P)$, where the $k$th path connects $u_k$ with $v_k$ for $k=1,\ldots n$. Then a lozenge tiling of $T(n,m;P)$ can be identified with an element of $\mathscr{P}(U,V(P))$. Note that this lattice path encoding is invertible: given a family of non-intersecting lattice paths starting at $u_k \in U$ and ending at $v_k \in V_P$ for $k=1,\ldots,n$, one can easily recover the lozenge tiling of $T(n,m;P)$. Indeed, there is a bijection between the set of lozenge tilings of $T(n,m;P)$ and $\mathscr{P}(U,V(P))$.
\begin{figure}[hbt!]
    \centering
    \subfigure[]{\label{fig:path1}\includegraphics[width=0.17\textwidth]{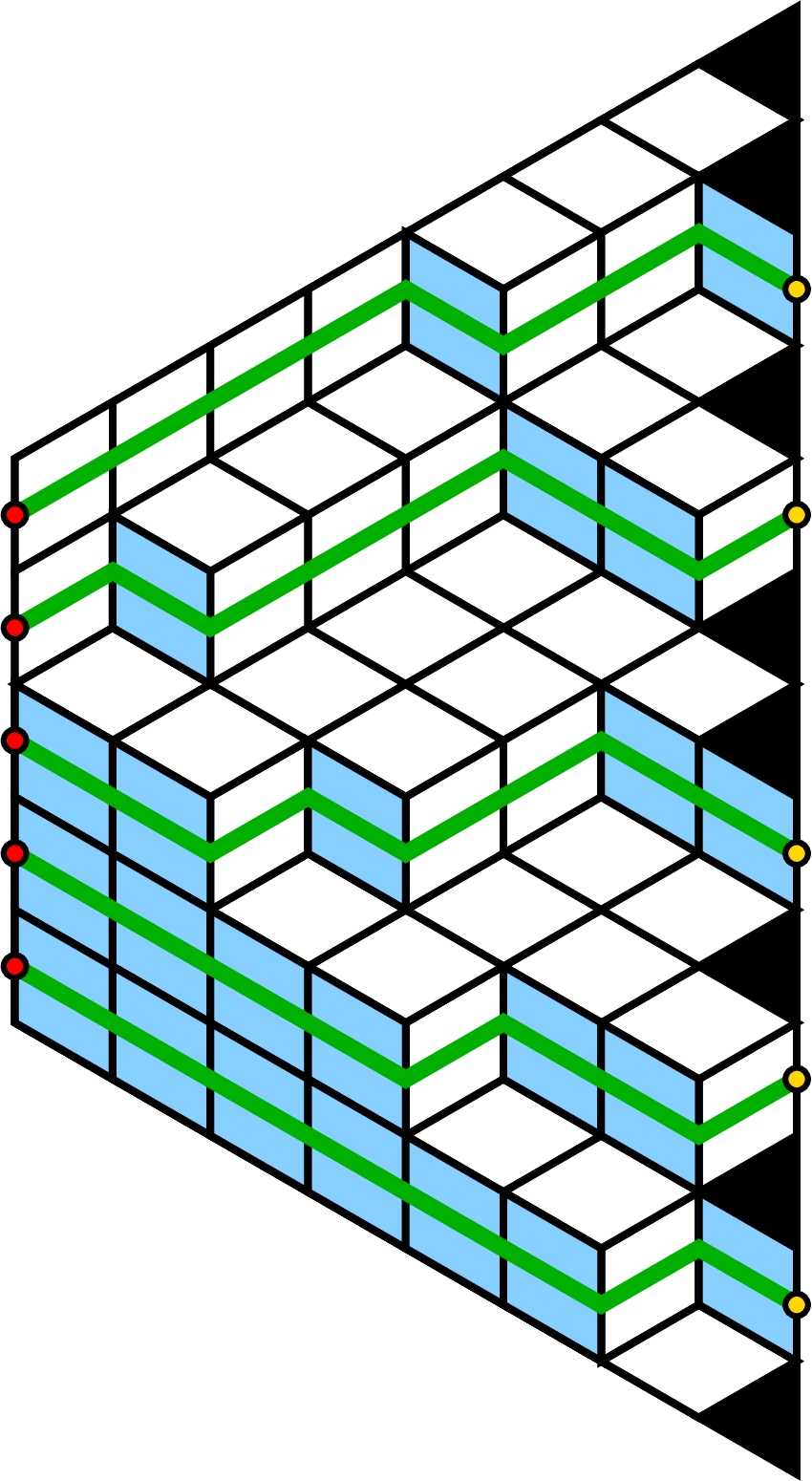}}
    \hspace{20mm}
    \subfigure[]{\label{fig:path2}\includegraphics[width=0.34\textwidth]{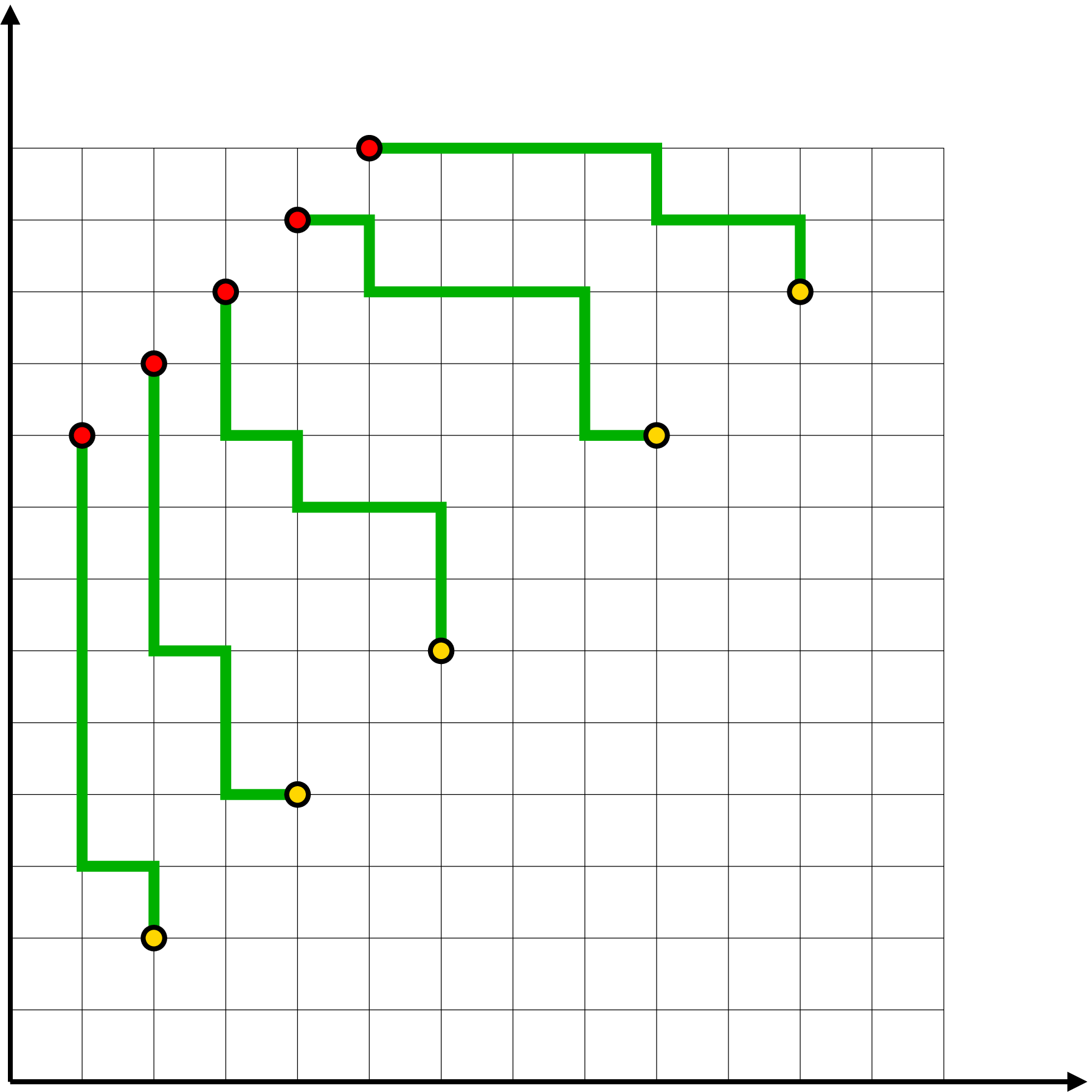}}
    \caption{(a) The lattice path encoding of the lozenge tiling given in Figure \ref{fig:blockTrap2}. (b) The family of non-intersecting lattice paths on $\mathbb{Z}^2$ that corresponds to Figure \ref{fig:path1}.}
\end{figure}

We remind the reader that in $T(n,m;P)$, horizontal and positive lozenges are weighted by $1$, negative lozenges are weighted by $q^{k}t$ if the distance from the right side of the region to that lozenge is $\frac{\sqrt{3}}{2}k$. Based on the lattice path encoding mentioned previously, one can observe that the bijection between the set of lozenge tilings of $T(n,m;P)$ and $\mathscr{P}(U,V(P))$ is weight preserving (note that lattice points $v_k$ are on $y=x$). Therefore, we can find the $(q,t)$-generating function of $T(n,m;P)$ for a fixed $n$-subset $P \subseteq [m+n]$ by enumerating the total weight of $\mathscr{P}(U,V(P))$, which is denoted by $\wt_{q,t}(\mathscr{P}(U,V(P)))$, using the LGV theorem.

The $(q,t)$-generating function of $\r$-block diagonally symmetric lozenge tilings of $H(n,m,m)$ is given by $\sum_{P}M_{q,t}(T(n,m;P)) = \sum_{P} \wt_{q,t}(\mathscr{P}(U,V(P)))$, where the sum runs over all the sets $P$ subject to the condition \eqref{eq.condition}. As mentioned at the beginning of the section, the LGV theorem cannot be applied directly due to the non-fixed ending points. On the other hand, Stembridge's generalization is not applicable since the selection of ending points is not entirely free. We resolve this issue by applying a simple trick to the lattice paths.
\begin{figure}[hbt!]
    \centering
    \subfigure[]{\label{fig:path3}\includegraphics[width=0.34\textwidth]{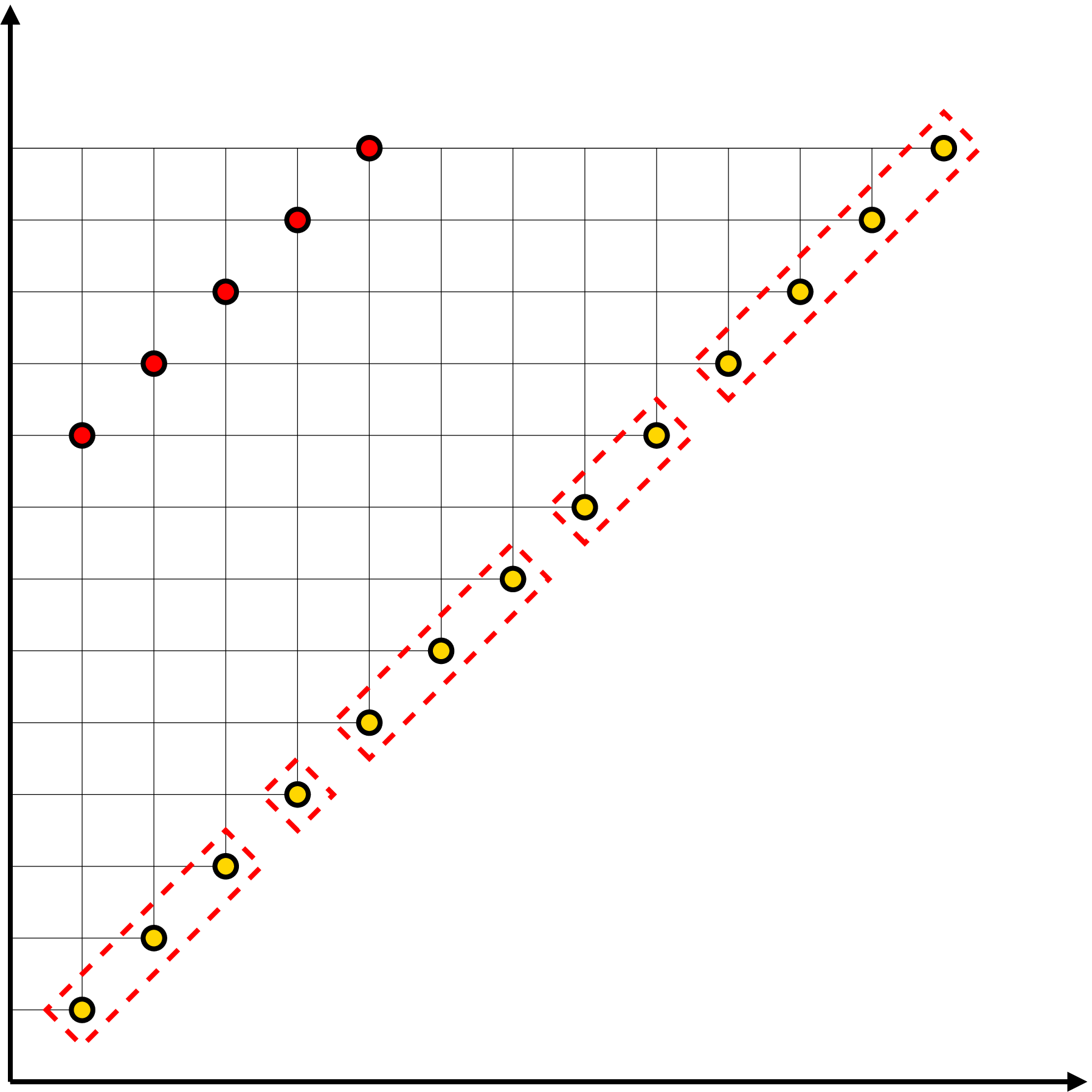}}
    \hspace{10mm}
    \subfigure[]{\label{fig:path4}\includegraphics[width=0.34\textwidth]{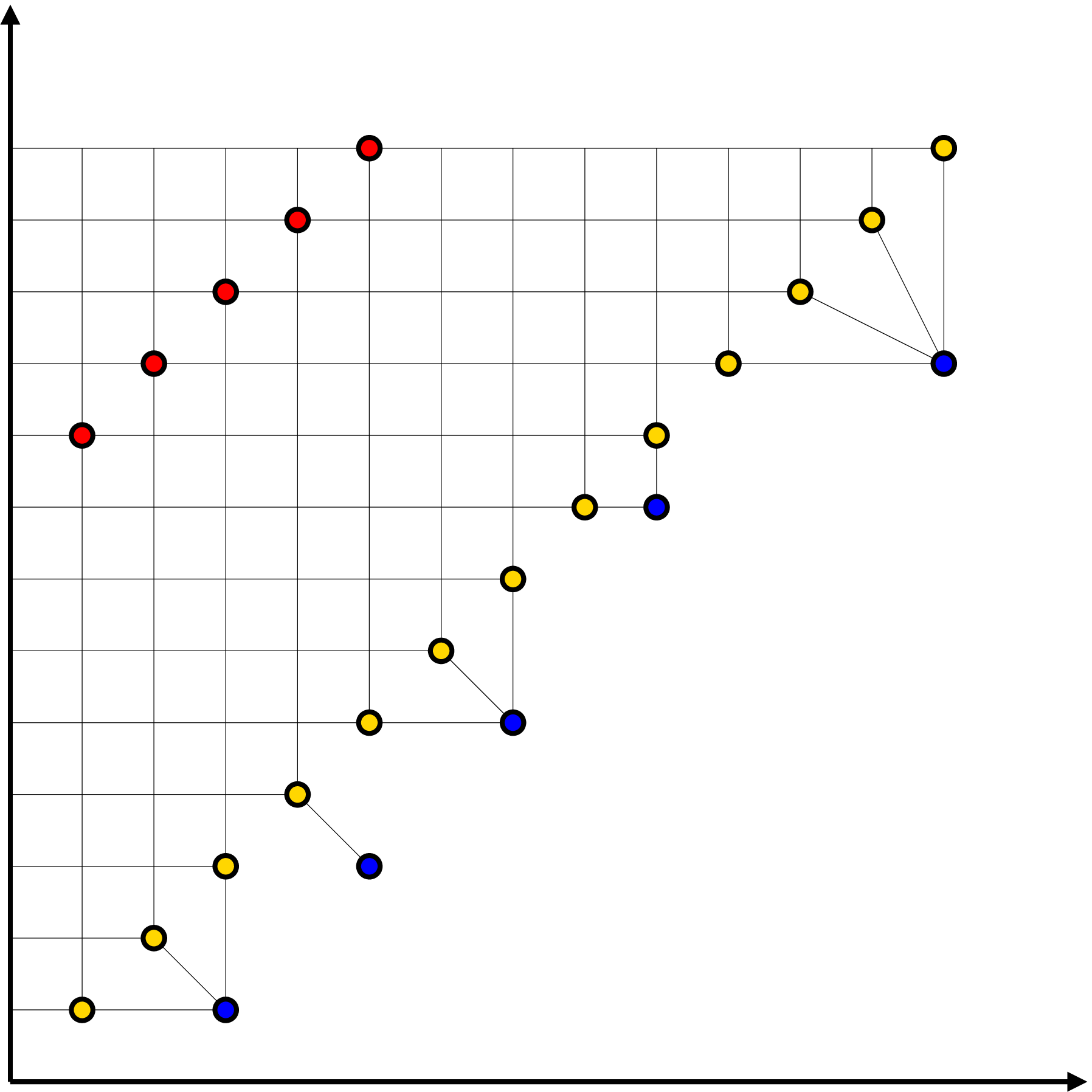}}
    \caption{(a) When $\mathbf{r}=(2,0,2,1,3)$, the five groups of ending points are shown in dotted boxes. (b) The five newly added points are indicated in blue.}
\end{figure}
\begin{proof}[First proof of Theorem \ref{thm:1}]
Notice that in the condition \eqref{eq.condition}, the set $[S_k+k]\setminus[S_{k-1}+k-1]$ has cardinality $r_k+1$. It is not hard to see that the elements of the complement of $P$, $[m+n]\setminus P=\{t_1,\ldots,t_n\}$, satisfy the conditions $t_k\in[S_k+k]\setminus[S_{k-1}+k-1]$ for each $k$. In other words, the element $t_k$ has $r_k+1$ choices from the set $[S_k+k]\setminus[S_{k-1}+k-1]$. Therefore, the weighted lozenge tilings that we considered are equivalent to the weighted families of non-intersecting lattice paths on $\mathbb{Z}^2$ that starts at $u_{k}=(k,m+k)$ and ends at $v_{k}=(t_k,t_k)$, where $t_k\in[S_k+k]\setminus[S_{k-1}+k-1]$ for $k=1,\ldots n$.

Now, the simple trick is presented as follows. We first partition the $m+n$ points $(1,1),\ldots,(m+n,m+n)$ into $n$ groups, the $k$th group $V_k$ consists of $r_k+1$ points $(S_{k-1}+k,S_{k-1}+k),\ldots, (S_k+k,S_k+k)$, for $k=1,\dots,n$. Figure \ref{fig:path3} shows how the ending points are grouped when $\mathbf{r}=(2,0,2,1,3)$. From the above discussion, the lattice path starting at $u_{k}$ should end at one of the points in $V_k$. We then add a new point $w_k$ at the bottom right of the $k$th group and $r_k+1$ edges joining this new point $w_k$ with each point in $V_k$. For convenience, let $W=\{w_1,\dots,w_n\}$ be the collection of these newly added points. These newly added edges are oriented toward $w_k$ and assigned a weight of $1$; see Figure \ref{fig:path4} for an illustration.

It is clear that any path going from $u_k$ to $w_k$ must pass through one of the points in $V_k$ for each $k=1,\ldots n$. Since the edges incident to $w_k$ is weighted by $1$, one can easily check that, for each $k=1,\ldots n$,
\begin{equation}\label{eq.pf1-1}
    \wt_{q,t}(u_k \rightarrow w_k) = \sum_{v_k} \wt_{q,t}(u_k \rightarrow v_k),
\end{equation}
where the sum runs over the points in $V_k$, that is, $v_{k}=(t_k,t_k)$ with $t_k\in[S_k+k]\setminus[S_{k-1}+k-1]$. Therefore, the total weight of $\mathscr{P}(U,W)$ can be expressed as
\begin{equation}\label{eq.pf1-3}
    \wt_{q,t}(\mathscr{P}(U,W))
      = \sum_{P} \wt_{q,t}(\mathscr{P}(U,V(P))) 
      = \sum_P \M_{q,t}(T(n,m;P)),  
\end{equation}
where the last two sums run over all the sets $P$ subject to the condition \eqref{eq.condition}. It remains to show that the total weight of our modified lattice paths is given by the product formula \eqref{eq:thm1.1}.

We can now apply the LGV theorem to enumerate $\wt_{q,t}(\mathscr{P}(U,W))$. This quantity is given by the determinant of the matrix $A = [a_{i,j}]_{i,j=1}^{n}$, where $a_{i,j}$ is given by the total weight of paths going from $u_i$ to $w_j$. Note that the weight on $\mathbb{Z}^2$ is translation invariant along the $y=x$ direction, this implies that
\begin{equation}\label{eq.pf1-4}
    \wt_{q,t}((x,y)\rightarrow(z,w))=\wt_{q,t}((x-z,y-z)\rightarrow(0,w-z))
\end{equation}
for all integers $x$, $y$, $z$, and $w$.

By Lemma \ref{lem:3.1}, \eqref{eq.pf1-1}, and \eqref{eq.pf1-4}, the $(i,j)$-entry of the matrix $A$ is given by
\begin{align}\label{eq:entry}
    a_{i,j} & = \wt_{q,t}(u_i \rightarrow w_j) \nonumber \\
    & =\sum_{k=0}^{r_j}\wt_{q,t}((i,m+i)\rightarrow(S_{j-1}+j+k,S_{j-1}+j+k)) \nonumber \\
    & = \sum_{k=0}^{r_j}\wt_{q,t}((i-(S_{j-1}+j+k),m+i-(S_{j-1}+j+k))\rightarrow(0,0)) \nonumber \\
    & =\sum_{k=0}^{r_j}q^{\frac{(m+i-(S_{j-1}+j+k+1))(m+i-(S_{j-1}+j+k))}{2}}t^{m+i-(S_{j-1}+j+k)}\begin{bmatrix} m\\m+i-(S_{j-1}+j+k) \end{bmatrix}_q. 
\end{align}

To evaluate the determinant of $A$, we apply a sequence of column and row operations to make the entries of $A$ simple. We first apply row operations by subtracting the $(n-k)$th row from the $(n-k+1)$th row of $A$ for $k=1,\ldots,n-1$ successively. The resulting matrix is denoted by $A^{\prime} = [a_{i,j}^{\prime}]$, where $a_{i,j}^{\prime}$ is given by
\begin{equation}
    a_{i,j}^{\prime}=
    \begin{cases}
        \sum_{k=0}^{r_j}q^{\frac{(m-(S_{j-1}+j+k))(m+1-(S_{j-1}+j+k))}{2}}t^{m+1-(S_{j-1}+j+k)}\begin{bmatrix} m\\m+1-(S_{j-1}+j+k) \end{bmatrix}_q & \text{if $i=1$,}\\
        \begin{aligned}
        &q^{\frac{(m+i-(S_{j-1}+j+1))(m+i-(S_{j-1}+j))}{2}}t^{m+i-(S_{j-1}+j)}\begin{bmatrix} m\\m+i-(S_{j-1}+j) \end{bmatrix}_q\\
        &-q^{\frac{(m+i-(S_{j}+j+2))(m+i-(S_{j}+j+1))}{2}}t^{m+i-(S_{j}+j+1)}\begin{bmatrix} m\\m+i-(S_{j}+j+1) \end{bmatrix}_q
        \end{aligned} & \text{if $i\geq2$.}
    \end{cases}
\end{equation}

The entries of $A^{\prime}$ are still complicated, so we apply column operations by adding the columns of $A^{\prime}$, ranging from the $(k+1)$th to the $n$th, to the $k$th column of $A^{\prime}$ (again, for $k=1,\ldots,n-1$ successively). The resulting matrix is denoted by $A^{\prime\prime} = [a_{i,j}^{\prime\prime}]$, where $a_{i,j}^{\prime\prime}$ is given by
\begin{equation}
\begin{aligned}
    a_{i,j}^{\prime\prime}&=
    \begin{cases}
        \sum_{k=0}^{r_j+\cdots+r_n+(n-j)}q^{\frac{(m-(S_{j-1}+j+k))(m+1-(S_{j-1}+j+k))}{2}}t^{m+1-(S_{j-1}+j+k)}\begin{bmatrix} m\\m+1-(S_{j-1}+j+k) \end{bmatrix}_q & \text{if $i=1$},\\
        \begin{aligned}
        &q^{\frac{(m+i-(S_{j-1}+j+1))(m+i-(S_{j-1}+j))}{2}}t^{m+i-(S_{j-1}+j)}\begin{bmatrix} m\\m+i-(S_{j-1}+j) \end{bmatrix}_q\\
        &-q^{\frac{(m+i-(S_{n}+n+2))(m+i-(S_{n}+n+1))}{2}}t^{m+i-(S_{n}+n+1)}\begin{bmatrix} m\\m+i-(S_{n}+n+1) \end{bmatrix}_q
        \end{aligned} & \text{if $i\geq2$,}
    \end{cases}
    \\
    &=\begin{cases}
        \sum_{k=0}^{r_j+\cdots+r_n+(n-j)}q^{\frac{(m-(S_{j-1}+j+k))(m+1-(S_{j-1}+j+k))}{2}}t^{m+1-(S_{j-1}+j+k)}\begin{bmatrix} m\\m+1-(S_{j-1}+j+k) \end{bmatrix}_q & \text{if $i=1$,}\\
        q^{\frac{(m+i-(S_{j-1}+j+1))(m+i-(S_{j-1}+j))}{2}}t^{m+i-(S_{j-1}+j)}\begin{bmatrix} m\\m+i-(S_{j-1}+j) \end{bmatrix}_q & \text{if $i\geq2$.}
    \end{cases}
\end{aligned}
\end{equation}
The second equality follows from the fact that $S_n=m$ and $\begin{bmatrix} m\\i-n-1 \end{bmatrix}_q=0$ for $i=2,\ldots,n$. Note that 
\begin{equation}
    a_{i,1}^{\prime\prime}=q^\frac{(m+i-2)(m+i-1)}{2}t^{(m+i-1)}\begin{bmatrix} m\\m+i-1 \end{bmatrix}_q=0
\end{equation}
for $i=2,\ldots,n$.

Furthermore, the $(1,1)$-entry of $A^{\prime\prime}$ simplifies to
\begin{equation}
\begin{aligned}
    a_{1,1}^{\prime\prime}&=\sum_{k=0}^{m+n-1}q^{\frac{(m-(1+k))(m+1-(1+k))}{2}}t^{m+1-(1+k)}\begin{bmatrix} m\\m+1-(1+k) \end{bmatrix}_q\\
    &=\sum_{k=0}^{m}q^{\frac{(m-(1+k))(m+1-(1+k))}{2}}t^{m+1-(1+k)}\begin{bmatrix} m\\m+1-(1+k) \end{bmatrix}_q\\
    &=\sum_{k=0}^{m}q^{\frac{(k-1)k}{2}}t^{k}\begin{bmatrix} m\\k \end{bmatrix}_q\\
    &=\prod_{i=1}^{m}(1+q^{i-1}t),
\end{aligned}
\end{equation}
where the final sum follows from the $q$-binomial theorem. Thus, we have
\begin{equation}\label{eq.pf1-4.5}
    \det A = \det A^{\prime}  = \det A^{\prime\prime}=\prod_{i=1}^{m}(1+q^{i-1}t) \cdot \det A^{\prime\prime}_{1,1},
\end{equation}
where $\det A^{\prime\prime}_{1,1}$ is the $(1,1)$-minor of $A^{\prime\prime}$. 

To complete the proof, it is enough to evaluate
\begin{equation}
    \det A^{\prime\prime}_{1,1}=\det\Bigg(q^{\frac{(m+i-(S_{j}+j+1))(m+i-(S_{j}+j))}{2}}t^{(m+i-(S_{j}+j))}\begin{bmatrix} m\\m+i-(S_{j}+j) \end{bmatrix}_q\Bigg)_{1\leq i,j\leq n-1}.
\end{equation}
We can pull out a factor $q^{\frac{i(i-1)}{2}}t^{i}$ from the $i$th row and $q^{\frac{(m-(S_{j}+j+1))(m-(S_{j}+j))}{2}}t^{m-(S_{j}+j)}$ from the $j$th column for each $i,j=1,\ldots,n-1$, and obtain
\begin{equation}\label{eq.pf1-5}
    \det A^{\prime\prime}_{1,1}=q^{\alpha_0}t^{\beta}\det\Bigg(q^{i(m-(S_{j}+j))}\begin{bmatrix} m\\m+i-(S_{j}+j) \end{bmatrix}_q\Bigg)_{1\leq i,j\leq n-1},
\end{equation}
where $\alpha_0=\sum_{i=1}^{n-1}\Big[\frac{i(i-1)}{2}+\frac{(m-(S_i+i+1))(m-(S_i+i))}{2}\Big]$ and $\beta=\sum_{i=1}^{n-1}(S_{n}-S_{i})=\sum_{i=1}^{n}(S_{n}-S_{i})$. Taking the transpose and rearranging entries of the matrix above, \eqref{eq.pf1-5} becomes
\begin{equation}\label{eq.pf1-6}
    \det A^{\prime\prime}_{1,1}=q^{\alpha_0}t^{\beta}\det\Bigg(q^{j(m-(S_{i}+i))}\begin{bmatrix} m\\m-(S_{i}+i)+j \end{bmatrix}_q\Bigg)_{1\leq i,j\leq n-1}.
\end{equation}

Finally, by Lemma \ref{lem:3.2} with $M=m$ and $L_i=m-(S_{i}+i)$, we are able to evaluate the determinant on the right-hand side of \eqref{eq.pf1-6}. Putting this result with \eqref{eq.pf1-4.5} together, and using the fact from a calculation that $\alpha_0+\sum_{i=1}^{n-1}i(m-(S_i+i))=\sum_{i=1}^{n-1}\binom{S_n-S_i}{2}=\sum_{i=1}^{n}\binom{S_n-S_i}{2}$, we obtain the desired result \eqref{eq:thm1.1}. This completes the proof of Theorem \ref{thm:1}.
\end{proof}

\section{Schur polynomials and the dual Pieri rule}\label{sec:algproof}

Schur polynomials are ubiquitous in mathematics, acting as a bridge between algebra, combinatorics, and geometry. Their rich combinatorial and representation-theoretic properties make them essential in various frameworks. In Section \ref{sec:Schur}, we review the definition of Schur polynomials and explain the relationship between $\r$-block diagonally symmetric lozenge tilings of a hexagon and Schur polynomials. In Section \ref{sec:Thm2.4secondproof}, we give a second proof of Theorem \ref{thm:1} using this relationship. In Section \ref{sec:Thm2.10proof}, we prove Theorem \ref{thm:2}.

\subsection{Skew Schur polynomials and lozenge tilings of trapezoidal regions with dents}\label{sec:Schur}

A \textit{partition} $\lambda$ of $n$ is a weakly decreasing sequence $(\lambda_1,\lambda_2,\dots,\lambda_k)$ of non-negative integers whose sum $|\lambda| := \lambda_{1}+\cdots+\lambda_{k}$ is $n$. By convention, we set $\lambda = \emptyset$ when $k=0$. When a partition $\lambda$ consists of $n_{i}$ copies of $i$ for $i=0,\ldots,s$, we abbreviate this partition by $\lambda=s^{n_{s}}\cdots0^{n_{0}}$.

The \textit{Young diagram} of shape $\lambda$ is a collection of boxes, arranged in left-justified rows, such that there are $\lambda_i$ boxes at row $i$. For another partition $\mu$, we write $\mu \subseteq \lambda$ whenever $\mu$ is contained in $\lambda$ as Young diagrams, that is, $\mu_i \leq \lambda_i$ for all $i$. In this case, we identify the Young diagram of the \textit{skew shape} $\lambda/\mu$ by removing the boxes of the Young diagram of $\lambda$ which are also boxes of the Young diagram of $\mu$.  A skew shape is called a \textit{horizontal strip} (resp., \textit{vertical strip}) if it contains no two boxes in the same column (resp., row). A \textit{$k$-horizontal strip} is a horizontal strip with exactly $k$ boxes, and similarly for vertical strips. 

A \textit{semistandard Young tableau} (SSYT) of shape $\lambda/\mu$ is a filling of the Young diagram of $\lambda / \mu$ with positive integers such that the numbers are weakly increasing along each row while strictly increasing along each column. Let $SSYT^m(\lambda/\mu)$ be the set of SSYTs of shape $\lambda/\mu$ with the entries at most $m$. One can associate a weight $\x^T = x_1^{T_1}x_2^{T_2}\cdots$ to an SSYT $T$ where $T_i$ is the number of times that $i$ appears in $T$. The \textit{skew Schur polynomial} $s_{\lambda/\mu}$ is defined to be 
\begin{equation}\label{eq.defskewschur}
    s_{\lambda/\mu}(x_1,\dots,x_m) = \sum_{T \in SSYT^m(\lambda/\mu)} \x^T.
\end{equation}
In particular, when $\mu = \emptyset$, \eqref{eq.defskewschur} reduces to the Schur polynomial $s_{\lambda}$. See \cite[Chapters 7.10 and 7.15]{Stanley} for more details.

Recall that $s_{\lambda/\mu}(x_1,\dots,x_m)$ is symmetric in the variables $x_1,\dots,x_m$. It follows directly from the definition that $s_{\lambda/\mu}(x_1,\ldots,x_m)$ is a homogeneous polynomial of degree $|\lambda|-|\mu|$. In other words, it satisfies
\begin{equation}\label{eq.schurp1}
    s_{\lambda/\mu}(tx_1,\ldots,tx_m)=t^{|\lambda|-|\mu|}s_{\lambda/\mu}(x_1,\ldots,x_m),
\end{equation}
for an indeterminate $t$. The \textit{principal specialization} of the Schur polynomial $s_{\lambda}(q^{0},\ldots,q^{m-1})$ is given by the following product formula \cite[Equation 7.105]{Stanley}:
\begin{align}
    \label{eq.schurp2}
    s_{\lambda}(q^0,\ldots,q^{m-1}) & =\frac{\prod_{1\leq i<j\leq m}(q^{\lambda_i+m-i}-q^{\lambda_j+m-j})}{\prod_{1\leq i<j\leq m}(q^{j-1}-q^{i-1})}.
\end{align}

The connection between (skew) Schur polynomials and lozenge tilings of $T(n+l,m;P,P^{\prime})$ introduced in Section \ref{sec:rrblock} has been studied by Ayyer and Fischer \cite{AF}. Their result about the tiling generating function of $T(n+l,m;P,P^{\prime})$ can be rephrased as follows\footnote{In fact, our weight assignment of lozenges is in the reverse order of the weight assignment given by Ayyer and Fischer \cite{AF}. Since skew Schur polynomials are symmetric, the tiling generating function of $T(n+l,m;P,P^{\prime})$ are invariant under permuting the weights.}.
\begin{theorem}{\cite[Theorem 6.1]{AF}}\label{thm.AF}
    For positive integers $m,n,l$, let $P=(p_1,\dots,p_{m+l})$ be a subset of $[m+n+l]$ and $P^{\prime}=(p_1^{\prime},\dots,p_{l}^{\prime})$ be a subset of $[n+l]$ such that elements in $P$ and $P^{\prime}$ are written in increasing order. Then
    \begin{equation}\label{eq.AFskew}
        \M_{\mathbf{x}}(T(n+l,m;P,P^{\prime})) = s_{\lambda(P)/\mu(P^{\prime})}(x_1,\dots,x_m),
    \end{equation}
    where $\lambda(P) = (p_{m+l}-(m+l),\dots,p_2-2,p_1-1)$ and $\mu(P^{\prime}) = (p^{\prime}_{l} - l,\dots,p^{\prime}_2-2,p^{\prime}_1-1)$ are partitions, which may contain $0$ as parts. 
\end{theorem}
In particular, if $l=0$, then $P^{\prime} = \emptyset$, $\mu(P^{\prime}) = \emptyset$, and \eqref{eq.AFskew} simplifies to the following expression, which can be found in \cite[Theorem 2.3]{AF}:
\begin{equation}\label{eq.AFschur}
    \M_{\mathbf{x}}(T(n,m;P)) = s_{\lambda(P)}(x_1,\dots,x_m),
\end{equation}
where $\lambda(P) = (p_{m}-m,\dots,p_2-2,p_1-1)$.

We characterize below the partitions $\lambda(P)$ where the sets $P \subset[m+n]$ satisfy the condition \eqref{eq.condition}. By \eqref{eq.xtrapezoidsum} and Theorem \ref{thm.AF}, this allows us to express the generating function of $\r$-block diagonally symmetric lozenge tilings of $H(n,m,m)$ as a sum of Schur polynomials. Recall from \eqref{eq.Rminmax}, among all the sets $P$ satisfying \eqref{eq.condition}, $P_{\mathsf{min}}$ (resp., $P_{\mathsf{max}}$) is the set of minimal (resp., maximal) labels. It is easy to see that
\begin{equation}
    \lambda(P_{\mathsf{min}}) = (n-1)^{r_n}(n-2)^{r_{n-1}}\cdots 1^{r_2}0^{r_1} \text{\quad and \quad} \lambda(P_{\mathsf{max}}) = n^{r_n}(n-1)^{r_{n-1}}\cdots 2^{r_2}1^{r_1}.
\end{equation}
See Figures \ref{fig:Ydia1} and \ref{fig:Ydia2} for examples.
\begin{figure}[hbt!]
    \centering
    \subfigure[]{\label{fig:Ydia1}\includegraphics[height=0.19\textwidth]{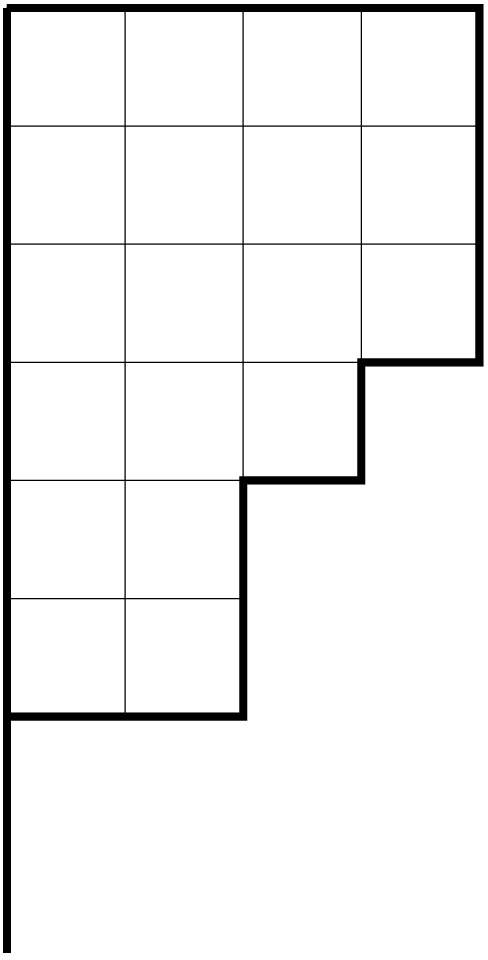}}
    \hspace{20mm}
    \subfigure[]{\label{fig:Ydia2}\includegraphics[height=0.19\textwidth]{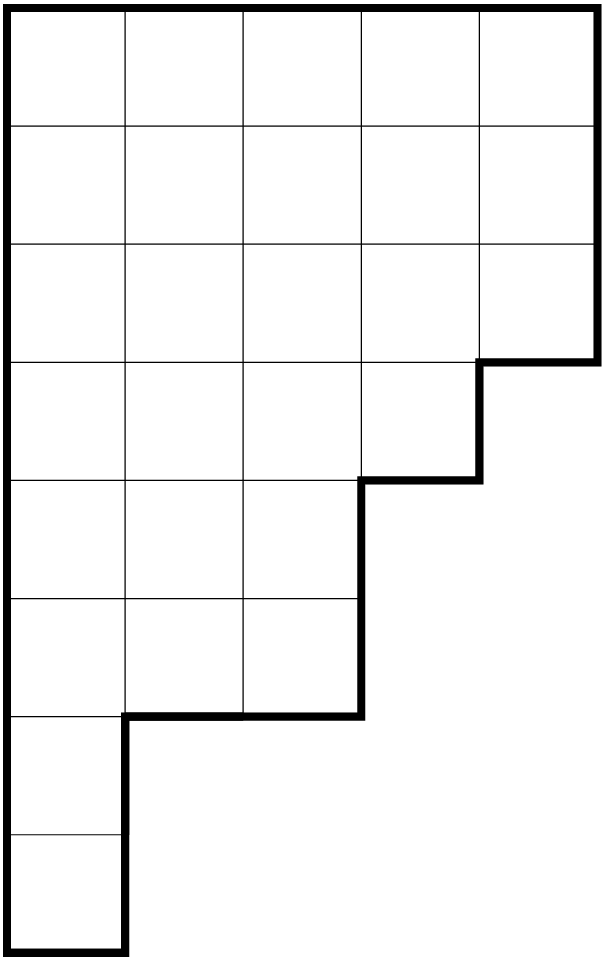}}
    \hspace{20mm}
    \subfigure[]{\label{fig:Ydia3}\includegraphics[height=0.19\textwidth]{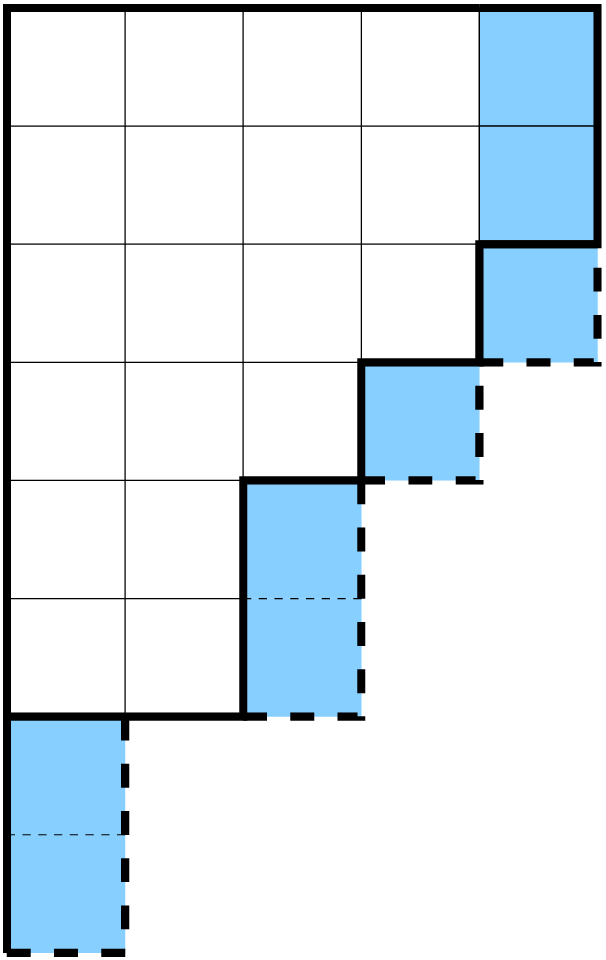}}
    \caption{When $\mathbf{r}=(2,0,2,1,3)$, the Young diagram of (a) the partition $\lambda(P_{\mathsf{min}})$, where $P_{\mathsf{min}}=\{1,2,5,6,8,10,11,12\}$; (b) the partition $\lambda(P_{\mathsf{max}})$, where $P_{\mathsf{max}}=\{2,3,6,7,9,11,12,13\}$; (c) the partition $\lambda(P)$ where $P=\{1,2,5,6,8,10,12,13\}$. The shaded boxes are the possible places to add vertically consecutive boxes.}
\end{figure}

For each $k$ in the condition \eqref{eq.condition}, the set $P$ contains $r_k$ elements selecting from $S_{k-1}+k,\dots,S_{k}+k-1,S_{k}+k$, which gives $r_k+1$ choices. Now, fix an integer $k, (1 \leq k \leq n)$, these $r_k+1$ choices correspond to partitions whose Young diagrams are obtained from $(k-1)^{r_k}$ by adding $0,1,\dots,r_k$ vertically consecutive boxes on the right, respectively. Figure \ref{fig:Ydia3} displays an example when $\mathbf{r}=(2,0,2,1,3)$, $k=5$, with the labels $\{10,12,13\}$ selected; the corresponding Young diagram is obtained from $\lambda(P_{\mathsf{min}})$ (Figure \ref{fig:Ydia1}) by adding two vertically consecutive boxes in the rightmost column. In other words, the condition \eqref{eq.condition} corresponds to the partitions $\lambda(P)$ satisfying $\lambda(P_{\mathsf{min}}) \subseteq \lambda(P) \subseteq \lambda(P_{\mathsf{max}})$. It is not hard to check that if $d(P) = i$ (i.e. the total difference between the labels in $P$ and $P_{\mathsf{min}}$; see \eqref{eq.difference}), then $\lambda(P)/\lambda(P_{\mathsf{min}})$ is an $i$-vertical strip. 

\subsection{The second proof of Theorem \ref{thm:1}}\label{sec:Thm2.4secondproof}

In the following lemma, we provide an auxiliary result about the product of the difference of each element in a given set, this will be used later to simplify the specialization of Schur polynomials.
\begin{lemma}\label{lem:product}
     Let $N$ be a positive integer, and let $I$ and $J$ be two disjoint sets such that $I \cup J = [N]$. Then
    \begin{equation}\label{eq.product}
        \prod_{\substack{1 \leq i_1 < i_2 \leq N \\ i_1,i_2 \in I}}[i_2-i_1]_q = \prod_{\substack{1 \leq j_1 < j_2 \leq N \\ j_1,j_2 \in J}}[j_2-j_1]_q \cdot \frac{\prod_{i=1}^{N-1}[i]_q!}{\prod_{j \in J}[j-1]_q! \cdot \prod_{j \in J}[N-j]_q!}.
    \end{equation}
\end{lemma}
\begin{proof}
    In these products, the indices always belong to $[N]$. We express $\prod_{1\leq a < b\leq N}[b-a]_q$ in two different ways to prove the lemma. First, it is obvious that
    \begin{equation}\label{eq.lem4-1}
        \prod_{1\leq a < b\leq N}[b-a]_q = \prod_{i=1}^{N-1}[i]_q!.
    \end{equation}
    Next, we separate the factors of the product on the left-hand side of \eqref{eq.lem4-1} into four groups: the factors where both $a$ and $b$ are in $I$, the factors where both are in $J$, the factors where just $a$ is in $I$, and the factors where just $a$ is in $J$. We then obtain
    \begin{align}
        \prod_{1\leq a < b\leq N}[b-a]_q & = \prod_{\substack{ i_1<i_2 \\ i_1,i_2 \in I}}[i_2-i_1]_q \cdot \prod_{\substack{j_1<j_2  \\ j_1,j_2 \in J}}[j_2-j_1]_q \cdot \prod_{\substack{ i<j  \\ i \in I,j \in J}}[j-i]_q \cdot \prod_{\substack{j<i \\ i \in I,j \in J}}[i-j]_q \nonumber \\
        & = \frac{\prod_{\substack{ i_1<i_2 \\ i_1,i_2 \in I}}[i_2-i_1]_q}{\prod_{\substack{j_1<j_2  \\ j_1,j_2 \in J}}[j_2-j_1]_q} \cdot \left(\prod_{\substack{j_1<j_2  \\ j_1,j_2 \in J}}[j_2-j_1]_q  \prod_{\substack{ i<j  \\ i \in I,j \in J}}[j-i]_q  \right) \left(\prod_{\substack{j_1<j_2  \\ j_1,j_2 \in J}}[j_2-j_1]_q \prod_{\substack{j<i \\ i \in I,j \in J}}[i-j]_q \right) \nonumber \\
        & = \frac{\prod_{\substack{ i_1<i_2 \\ i_1,i_2 \in I}}[i_2-i_1]_q}{\prod_{\substack{j_1<j_2  \\ j_1,j_2 \in J}}[j_2-j_1]_q} \cdot \prod_{a<j,\ j \in J}[j-a]_q \cdot \prod_{j<b,\ j \in J}[b-j]_q \nonumber \\
        & = \frac{\prod_{\substack{ i_1<i_2 \\ i_1,i_2 \in I}}[i_2-i_1]_q}{\prod_{\substack{j_1<j_2  \\ j_1,j_2 \in J}}[j_2-j_1]_q} \cdot \prod_{j \in J}[j-1]_q! \cdot \prod_{j \in J}[N-j]_q!. \label{eq.lem4-2}
    \end{align}

    Combining \eqref{eq.lem4-1} and \eqref{eq.lem4-2} yields the desired result.
\end{proof}

The second proof of Theorem \ref{thm:1} is presented below.
\begin{proof}[Second proof of Theorem \ref{thm:1}]
    Based on the above characterization of $\lambda(P)$, by \eqref{eq.xtrapezoidsum} and \eqref{eq.AFschur},
    \begin{equation}\label{eq.pf2-1}
        \M_{\x}^{\mathbf{r}} (H(n,m,m))  = \sum_{P} \M_{\x}(T(n,m;P)) = \sum_{\lambda(P)}s_{\lambda(P)}(x_1,\dots,x_m),
    \end{equation}
        where the first sum runs over the sets $P \subseteq [m+n]$ satisfying \eqref{eq.condition} and the last sum is over the partitions $\lambda(P)$ satisfying $\lambda(P)/\lambda(P_{\mathsf{min}})$ is a vertical strip.

    To evaluate the sum \eqref{eq.pf2-1}, we recall the dual Pieri rule (see for instance \cite[Page 340]{Stanley}): for any partition $\lambda$ and positive integer $i$,
    \begin{equation}\label{eq.dualPR}
        s_{\lambda}(x_1,\dots,x_m)e_i(x_1,\dots,x_m) = \sum_{\text{$\lambda^{+}/\lambda$ $i$-vert. strip}}s_{\lambda^{+}}(x_1,\dots,x_m),
    \end{equation}
    where $\displaystyle e_i(x_1,\dots,x_m) = \sum_{1\leq j_1<j_2<\cdots<j_i\leq m}x_{j_1}x_{j_2}\cdots x_{j_i}$ is the \textit{$i$th elementary symmetric polynomial} in $m$ variables. Therefore, by the dual Pieri rule, we have
    \begin{equation}\label{eq.pf2-2}
        s_{\lambda(P_{\mathsf{min}})}(x_1,\dots,x_m)e_i(x_1,\dots,x_m) = \sum_{\{\lambda(P)| d(P)=i\}}s_{\lambda(P)}(x_1,\dots,x_m).
    \end{equation}
    Summing over all positive integers $i$ on both sides of \eqref{eq.pf2-2}, we obtain
    \begin{equation}\label{eq.pf2-3}
        s_{\lambda(P_{\mathsf{min}})}(x_1,\dots,x_m) \left( \sum_{i=0}^{m}e_i(x_1,\dots,x_m) \right) = \sum_{\lambda(P)}s_{\lambda(P)}(x_1,\dots,x_m).
    \end{equation}

    Thus, to obtain the product formula for the $(q,t)$-generating function, we specialize the weight $x_k = q^{k-1}t$ in \eqref{eq.pf2-3}. It remains to compute the left-hand side of \eqref{eq.pf2-3} under this specialization. For convenience, the partition $\lambda(P_{\mathsf{min}})$ is denoted by $\lambda^0$. Note that the sum of elementary symmetric polynomials factors into
    \begin{equation}\label{eq.sumofeifactor}
        \sum_{i=0}^{m}e_i(x_1,\dots,x_m)=\prod_{i=1}^{m}(1+x_i) = \prod_{i=1}^{m}(1+q^{i-1}t).
    \end{equation}
    Since the Schur polynomial $s_{\lambda^0}(x_1,\dots,x_m)$ is homogeneous of degree $|\lambda^0|=\sum_{i=1}^{n}(S_{n}-S_{i})$, then \eqref{eq.schurp1} implies that
    \begin{equation}\label{eq.pf2-5}
        s_{\lambda^0}(q^{0}t,\ldots,q^{m-1}t) = t^{\sum_{i=1}^{n}(S_{n}-S_{i})}s_{\lambda^0}(q^{0},\ldots,q^{m-1}).
    \end{equation}
    The principal specialization of Schur polynomials \eqref{eq.schurp2} gives the following expression:
    \begin{align}
        s_{\lambda^0}(q^{0},\ldots,q^{m-1})&=\frac{\prod_{1\leq i<j\leq m}(q^{\lambda^0_i+m-i}-q^{\lambda^0_j+m-j})}{\prod_{1\leq i<j\leq m}(q^{j-1}-q^{i-1})} \nonumber\\
        &=\frac{\prod_{1\leq i<j\leq m}q^{\lambda^0_j+m-j}\left( q^{(\lambda^0_i+m-i+1)-(\lambda^0_j+m-j+1)}-1 \right)}{\prod_{1\leq i<j\leq m}q^{i-1}(q^{j-i}-1)} \nonumber\\
        &=\frac{q^{\sum_{1\leq i<j\leq m}(\lambda^0_j+m-j)}}{q^{\sum_{1\leq i<j\leq m}(i-1)}}\cdot\frac{\prod_{1\leq i<j\leq m}[(\lambda^0_i+m-i+1)-(\lambda^0_j+m-j+1)]_q}{\prod_{1\leq i<j\leq m}[j-i]_q}. \label{eq.pf2-6}
    \end{align}

    One can readily check that $\sum_{1\leq i<j\leq m}(m-j)=\frac{m(m-1)(m-2)}{6}=\sum_{1\leq i<j\leq m}(i-1)$. Since $\lambda^0 = (n-1)^{r_n}(n-2)^{r_{n-1}}\ldots1^{r_2}0^{r_1}$, we may write $\lambda^0_j = \sum_{i=1}^n\chi(\lambda^0_j \geq i)$, where $\chi(P) = 1$ if the statement $P$ is true and $\chi(P) = 0$ otherwise. Then we have
    \begin{equation*}
       \sum_{1\leq i<j\leq m}\lambda^0_j = \sum_{j=1}^{m}(j-1)\lambda^0_j = \sum_{j=1}^{m}\sum_{i=1}^n (j-1) \chi(\lambda^0_j \geq i) = \sum_{i=1}^n \sum_{j=1}^{m} (j-1) \chi(\lambda^0_j \geq i) =
       \sum_{i=1}^{n}\sum_{j=1}^{S_n-S_i}(j-1)=\sum_{i=1}^{n}\binom{S_n-S_i}{2}.
    \end{equation*}
    Thus, the first term in the last line of \eqref{eq.pf2-6}, which is a quotient of powers of $q$, simplifies to $q^{\alpha}$ where $\alpha = \sum_{i=1}^n \binom{S_n-S_i}{2}$.

    For the second term in the last line of \eqref{eq.pf2-6}, the denominator can be rewritten as $\prod_{1\leq i<j\leq m}[j-i]_q = \prod_{i=1}^{m-1}[i]_q!$. One can check that\footnote{Note that $\lambda^0_{i}+m-i+1$ and $\lambda^0_{i+1}+m-(i+1)+1$ are differ by $1$ if and only if $i\notin\{S_{n}-S_{n-1}, S_{n}-S_{n-2},\ldots, S_{n}-S_{1}\}$, because $\lambda^0_{i}=\lambda^0_{i+1}$ precisely when $i$ is excluded from this set. Moreover, when $i\in\{S_{n}-S_{n-1}, S_{n}-S_{n-2},\ldots, S_{n}-S_{1}\}$, $\lambda^0_{i}+m-i+1$ and $\lambda^0_{i+1}+m-(i+1)+1$ differ by $2$ and the numbers we skip are $S_{i}+i$, for $i=1,\ldots,n-1$.} $\{\lambda^0_{i}+m-i+1~|~i=1,\dots,m\} = [S_n+n]\setminus \{S_k+k~|~k=1,\dots,n\}$. The numerator can be obtained from Lemma \ref{lem:product} by setting $N = S_n+n$, $I=[S_n+n]\setminus \{S_k+k|k=1,\dots,n\}$, and $J=\{S_k+k~|~k=1,\dots,n\}$. Therefore,
    \begin{equation}\label{eq.pf2-8}
        \prod_{1\leq i<j\leq m}[(\lambda^0_{i}+m-i+1)-(\lambda^0_{j}+m-j+1)]_q = \prod_{1\leq i<j\leq n}[S_j-S_i+j-i]_q \cdot \frac{\prod_{i=1}^{S_n+n-1}[i]_q!}{\prod_{i=1}^{n}[S_i+i-1]_q!\prod_{i=1}^{n}[S_n-S_i+n-i]_q!}.
    \end{equation}

    Finally, combining \eqref{eq.sumofeifactor}, \eqref{eq.pf2-5}, \eqref{eq.pf2-6}, and \eqref{eq.pf2-8}, we obtain the desired result.
\end{proof}
\begin{remark}\label{rmk.asm}
    As mentioned at the end of Section \ref{sec:rblocklozenge}, when we specialize $r_1= \cdots =r_n =2$, the number of $\r$-block diagonally symmetric lozenge tilings of $H(n,2n,2n)$ coincides with the number of ASMs of size $n$ up to some factors. In fact, this connection came from an algebraic expression of the number of ASMs as characters of classical groups due to Okada \cite[Theorem 1.2]{Okada06} (see also \cite[Section 2]{Okada06} for definitions). He showed that
    \begin{equation}
        \prod_{k=0}^{n-1}\frac{(3k+1)!}{(n+k)!} = 3^{-n(n-1)/2} \cdot \dim \mathrm{GL}_{2n}(\delta),
    \end{equation}
    where $\dim \mathrm{GL}_{2n}(\delta)$ denotes the dimension of the irreducible representation of $\mathrm{GL}_{2n}$ with highest weight $\delta=(n-1,n-1,n-2,n-2,\dots,1,1,0,0)$. Note that this is given by the Schur polynomial $s_{\delta}(1,1,\dots,1)$ (see for instance \cite[Chapter 7, Appendix 2]{Stanley}). On the other hand, when $r_i=2$ for all $i$, the set $P_{\mathsf{min}} = [3n]\setminus \{3k|k=1,\dots,n\}$ and the corresponding partition $\lambda(P_{\mathsf{min}})$ is equal to $\delta$. By \eqref{eq.pf2-1} and \eqref{eq.pf2-3},
    \begin{equation}
        \M^{\r}(H(n,2n,2n)) = 2^{2n} \cdot s_{\delta}(1,1,\dots,1).
    \end{equation}
    This explains why the ASM numbers appear in our special case algebraically. However, finding a direct bijection between ASMs and our block diagonally symmetric lozenge tilings remains open.
\end{remark}

\subsection{Proof of Theorem \ref{thm:2}}\label{sec:Thm2.10proof}

We remind the reader from \eqref{eq.traptwodentssum} that $(\r,\r^{\prime})$-block diagonally symmetric lozenge tilings of $H(n+l,m,m)$ correspond to lozenge tilings of regions $T(n+l,m;P,P^{\prime})$ where the sets $P \subset [m+n+l]$ satisfy \eqref{eq.condition} while $P^{\prime} \subset [n+l]$ satisfy \eqref{eq.condition2}. By Theorem \ref{thm.AF}, we can express weighted lozenge tilings of $T(n+l,m;P,P^{\prime})$ as a skew Schur polynomial associated with the partition $\lambda(P)/\mu(P^{\prime})$. To evaluate the signed sum of these skew Schur polynomials over the sets $P$ and $P^{\prime}$ subject to the conditions mentioned above, we need the following dual Pieri rule for skew Schur polynomials due to Assaf and McNamara \cite{AM11}.
\begin{theorem}{\cite[Corollary 3.3]{AM11}}\label{thm.AM}
    For any skew shape $\lambda/\mu$ and any positive integer $i$, we have
    \begin{equation}\label{eq.skewpierirule}
        s_{\lambda/\mu}(x_1,\dots,x_m)e_i(x_1,\dots,x_m) = \sum_{k=0}^{i}(-1)^k \sum_{\substack{\lambda^{+}/\lambda \text{ $(i-k)$-vert. strip}\\ \mu/\mu^{-} \text{ $k$-hor. strip}}}s_{\lambda^{+}/\mu^{-}}(x_1,\dots,x_m),
    \end{equation}
where the sum is over all partitions $\lambda^{+}$ and $\mu^{-}$ such that $\lambda^{+}/\lambda$ is a vertical strip of size $i-k$ and $\mu/\mu^{-}$ is a horizontal strip of size $k$.
\end{theorem}

Now, we are ready to prove Theorem \ref{thm:2}.
\begin{proof}[Proof of Theorem \ref{thm:2}]
    The characterization of partitions $\lambda(P)$ is the same as before described in Section \ref{sec:Schur}, that is, if $d(P) = i$, then $\lambda(P)/\lambda(P_{\mathsf{min}})$ is an $i$-vertical strip. Similarly, for the sets $P^{\prime}$ satisfying \eqref{eq.condition2}, we have $\mu(P^{\prime}_{\mathsf{min}}) \subseteq \mu(P^{\prime}) \subseteq \mu(P^{\prime}_{\mathsf{max}})$. If $d^{\prime}(P^{\prime}) = i$ (i.e. the total difference between the labels in $P^{\prime}$ and $P^{\prime}_{\mathsf{max}}$; see \eqref{eq.Rminmax}), then $\mu(P^{\prime}_{\mathsf{max}})/\mu(P^{\prime})$ is an $i$-vertical strip.
    
    However, by the assumption that the tuple $\r^{\prime} = (r_1^{\prime},\dots,r_n^{\prime})$ satisfies $r_i^{\prime} = 1$ for $i=1,\dots,l$ and $r^{\prime}_i = 0$ for $i=l+1,\dots,n$, the set $P^{\prime}_{\mathsf{max}} = \{2,4,6,\dots,2l\}$. This implies that the partition $\mu(P^{\prime}_{\mathsf{max}}) = (l,l-1,\dots,1)$ is a staircase shape. Thus, $\mu(P^{\prime}_{\mathsf{max}})/\mu(P^{\prime})$ is indeed a union of single boxes in distinct columns; therefore, it is also a horizontal strip.
    
    By \eqref{eq.trapetwodentssign} and Theorem \ref{thm.AF},
    \begin{align}
        \M^{\r,\r^{\prime}}_{\x,\sgn}\left( H(n+l,m,m) \right) & = \sum_{P,P^{\prime}} (-1)^{d^{\prime}(P^{\prime})}\M_{\x}\left( T(n+l,m;P,P^{\prime}) \right) \nonumber \\
        & = \sum_{\lambda(P),\mu(P^{\prime})} (-1)^{d^{\prime}(P^{\prime})} s_{\lambda(P)/\mu(P^{\prime})}(x_1,\dots,x_m), \label{eq.pf3-1}       
    \end{align}
    where the first sum runs over the sets $P \subset [m+n+l]$ satisfying \eqref{eq.condition} and the sets $P^{\prime} \subset [n+l]$ satisfying \eqref{eq.condition2}; the last sum runs over the partitions $\lambda(P)$ such that $\lambda(P)/\lambda(P_{\mathsf{min}})$ is a vertical strip and the partitions $\mu(P)$ satisfying $\mu(P^{\prime}_{\mathsf{max}})/\lambda(P^{\prime})$ is a vertical strip (also a horizontal strip).

    By Theorem \ref{thm.AM}, we have the following expression:
    \begin{equation}\label{eq.pf3-2}
        s_{\lambda(P_{\mathsf{min}})/\mu(P^{\prime}_{\mathsf{max}})}(x_1,\dots,x_m)e_i(x_1,\dots,x_m) = \sum_{k=0}^{i}(-1)^k \sum_{\substack{\lambda(P)/\lambda(P_{\mathsf{min}}) \text{ $(i-k)$-vert. strip}\\ \mu(P^{\prime}_{\mathsf{max}})/\mu(P^{\prime}) \text{ $k$-hor. strip}}}s_{\lambda(P)/\mu(P^{\prime})}(x_1,\dots,x_m).
    \end{equation}
    Next, we sum over all positive integers $i$ on both sides of \eqref{eq.pf3-2}. The left-hand side yields
    \begin{equation}\label{eq.pf3-3}
        s_{\lambda(P_{\mathsf{min}})/\mu(P^{\prime}_{\mathsf{max}})}(x_1,\dots,x_m) \left( \sum_{i=0}^{m} e_i(x_1,\dots,x_m) \right) = \M_{\x}(T(n+l,m;P_{\mathsf{min}},P^{\prime}_{\mathsf{max}}))\cdot \prod_{i=1}^m(1+x_i).
    \end{equation}
    On the right-hand side, if we interchange the order of summation (from $\sum_{i\geq0}\sum_{k=0}^i$ to $\sum_{k\geq0}\sum_{i\geq  k}$), then
    \begin{equation}\label{eq.pf3-4}
        \sum_{k\geq 0}(-1)^k \sum_{\substack{\lambda(P)/\lambda(P_{\mathsf{min}}) \text{ vert. strip}\\ \mu(P^{\prime}_{\mathsf{max}})/\mu(P^{\prime}) \text{ $k$-hor. strip}}}s_{\lambda(P)/\mu(P^{\prime})}(x_1,\dots,x_m) =  \sum_{\lambda(P),\mu(P^{\prime})}  (-1)^{d^{\prime}(P^{\prime})} s_{\lambda(P)/\mu(P^{\prime})}(x_1,\dots,x_m),
    \end{equation}
    which is equal to \eqref{eq.pf3-1}. Finally, combining \eqref{eq.pf3-1}, \eqref{eq.pf3-4}, and \eqref{eq.pf3-3}, we obtain the desired identity.
\end{proof}

\section*{Acknowledgments}
The authors would like to thank Mihai Ciucu for his interest in this work and insightful suggestions. The authors also thank anonymous referees for carefully reading the manuscript and giving useful comments.

\end{document}